\documentclass[12pt]{amsart}
\usepackage{amscd,verbatim}
\usepackage{times}
\usepackage[all]{xy}
\usepackage{amssymb}

\newcommand{\Ker}{\operatorname{Ker}}
\newcommand{\IM}{\operatorname{Im}}
\newcommand{\Coker}{\operatorname{Coker}}

\newcommand{\Spec}{\operatorname{Spec}}
\newcommand{\Sm}{\operatorname{Sm}}
\newcommand{\Cor}{\operatorname{Cor}}
\newcommand{\car}{\operatorname{char}}

\newcommand{\Pic}{\operatorname{Pic}}
\newcommand{\NS}{\operatorname{NS}}
\newcommand{\Alb}{\operatorname{Alb}}
\newcommand{\uHom}{\operatorname{\underline{Hom}}}
\newcommand{\Hom}{\operatorname{Hom}}
\newcommand{\End}{\operatorname{End}}
\newcommand{\Tor}{\operatorname{Tor}}
\newcommand{\Ext}{\operatorname{Ext}}
\newcommand{\Ab}{\operatorname{Ab}}
\newcommand{\DA}{\operatorname{DA}}
\newcommand{\DM}{\operatorname{DM}}
\newcommand{\SH}{\operatorname{SH}}
\newcommand{\NST}{\operatorname{NST}}
\newcommand{\rat}{{\operatorname{rat}}}
\newcommand{\num}{{\operatorname{num}}}
\newcommand{\eff}{{\operatorname{eff}}}
\newcommand{\gm}{{\operatorname{gm}}}
\newcommand{\cont}{{\operatorname{cont}}}
\newcommand{\tors}{{\operatorname{tors}}}

\newcommand{\Nis}{{\operatorname{Nis}}}
\newcommand{\et}{{\operatorname{\acute{e}t}}}
\newcommand{\ord}{\operatorname{ord}}
\newcommand{\rk}{\operatorname{rk}}
\newcommand{\cl}{\operatorname{cl}}
\newcommand{\cone}{\operatorname{cone}}
\renewcommand{\lim}{\varprojlim}
\newcommand{\colim}{\varinjlim}
\newcommand{\hocolim}{\operatorname{hocolim}}

\newcommand{\Inj}{\lhook\joinrel\longrightarrow}
\newcommand{\surj}{\rightarrow\!\!\!\!\!\rightarrow}
\newcommand{\Surj}{\relbar\joinrel\surj}
\newcommand{\by}[1]{\overset{#1}{\longrightarrow}}
\newcommand{\iso}{\by{\sim}}
\newcommand{\yb}[1]{\overset{#1}{\longleftarrow}}
\newcommand{\osi}{\yb{\sim}}
\newcommand{\oo}{\mathop{\otimes}\limits}

\newcommand{\sA}{\mathcal{A}}

\newcommand{\sC}{\mathcal{C}}
\newcommand{\sF}{\mathcal{F}}
\newcommand{\sH}{\mathcal{H}}
\newcommand{\sM}{\mathcal{M}}
\newcommand{\sN}{\mathcal{N}}
\newcommand{\sS}{\mathcal{S}}
\newcommand{\sT}{\mathcal{T}}
\newcommand{\sX}{\mathcal{X}}
\newcommand{\sZ}{\mathcal{Z}}

\newcommand{\A}{\mathbf{A}}

\newcommand{\F}{\mathbf{F}}
\newcommand{\G}{\mathbb{G}}
\renewcommand{\L}{\mathbb{L}}

\renewcommand{\P}{\mathbf{P}}
\newcommand{\Q}{\mathbf{Q}}
\newcommand{\R}{\mathbf{R}}
\newcommand{\Z}{\mathbf{Z}}
\newcommand{\un}{\mathbf{1}}

\newcommand{\bH}{\mathbb{H}}

\renewcommand{\phi}{\varphi}

\DeclareFontFamily{U}{wncy}{}
\DeclareFontShape{U}{wncy}{m}{n}{%
<5>wncyr5%
<6>wncyr6%
<7>wncyr7%
<8>wncyr8%
<9>wncyr9%
<10>wncyr10%
<11>wncyr10%
<12>wncyr6%
<14>wncyr7%
<17>wncyr8%
<20>wncyr10%
<25>wncyr10}{}
\DeclareMathAlphabet{\cyr}{U}{wncy}{m}{n}

\newtheorem{Th}{Theorem}
\newtheorem{thm}{Theorem}[subsection]
\newtheorem{prop}[thm]{Proposition}
\newtheorem{lemma}[thm]{Lemma}
\newtheorem{cor}[thm]{Corollary}
\newtheorem{conj}[thm]{Conjecture}
\theoremstyle{definition}
\newtheorem{defn}[thm]{Definition}

\newtheorem{remark}[thm]{Remark}
\newtheorem{remarks}[thm]{Remarks}

\newcounter{spec}
\newenvironment{thlist}{\begin{list}{\rm{(\roman{spec})}}%
{\usecounter{spec}\labelwidth=20pt\itemindent=0pt\labelsep=10pt}}%
{\end{list}}%

\numberwithin{equation}{subsection}
\begin{document}

\title{The full faithfulness conjectures in characteristic $p$}
\author{Bruno Kahn}
\address{Institut de Math\'ematiques de Jussieu\\Case 247\\
4 place Jussieu\\
75252 Paris Cedex 05\\France}
\email{kahn@math.jussieu.fr}
\date{September 19, 2012}
\begin{abstract}
We present a triangulated version of the conjectures of Tate and Beilinson on algebraic cycles
over a finite field. This sheds a new light on Lichtenbaum's Weil-\'etale cohomology.
\end{abstract}
\maketitle

\tableofcontents

\section*{Introduction} 

It is generally understood that the ``standard" conjectures on mixed motives predict
that certain triangulated realisation functors should be conservative. The aim of this text is
to explain that, at least in characteristic $p$, they predict much more: namely, that suitable
triangulated realisation functors should be \emph{fully faihtful}.

The main result is the following. Let $\F$ be a finite field, and let $\DM_\et(\F)$ be
Voevodsky's stable category of (unbounded, \'etale) motivic complexes. It contains the category
$\DM_\et^\eff(\F)$ of effective motivic complexes as a full subcategory. Let
$l$ be a prime number different from $2$ and $\car \F$. By work of Ayoub \cite{ayoubreal}, there is a
pair of adjoint functors:
\[\DM_\et(\F)\begin{smallmatrix}\Omega_l\\\displaystyle \leftrightarrows\\ R_l
\end{smallmatrix}\hat{D}_\et(\F,\Z_l)\]
where the right hand side is Ekedahl's category of $l$-adic coefficients \cite{ek}. In
particular, we have the object
\[\Gamma=\Omega_l(\Z_l)\in \DM_\et(\F).\]

\begin{Th}[cf. Corollary \protect{\ref{c9.1}}] \label{tmain}The following conditions are equivalent:
\begin{thlist}
\item The Tate conjecture (on the poles of the zeta function) and the Beilinson conjecture (on
rational equivalence agreeing with numerical equivalence) hold for any smooth projective
$\F$-variety.
\item $\Gamma\in \DM_\et^\eff(\F)$.
\end{thlist}
\end{Th}

The full faithfulness statement announced above appears as another equivalent condition in
Proposition \ref{p10.1} a); further equivalent conditions (finite generation of Hom groups)
appear in Theorem \ref{t10.1}. After the fact, see \S \ref{s.w}, these reformulations involve
Weil \'etale cohomology. For a case when they hold in the triangulated context, see
Theorem \ref{t4ab}.

Curiously, the Beilinson conjecture and the Parshin conjecture (on vanishing of higher rational
$K$-groups of smooth projective $\F$-varieties) are sufficient to imply the existence of a
motivic $t$-structure on $\DM_\gm(\F,\Q)$, as well as semi-simplicity and independence of $l$
for the $\Q_l$-adic realisations of objects of this category: see Proposition \ref{p10.2}.

A problem is that there is no known analogue of this picture in characteristic $0$ at the
moment. While in characteristic $p$ a single $l$-adic cohomology is sufficient to approach
cycles modulo rational equivalence, it seems that in characteristic $0$ one should consider the
full array of realisation functors, plus their comparison isomorphisms. Even with this idea it
does not seem obvious how to get a clean conjectural statement. In the light of \S \ref{s.w},
this might be of great interest to get the right definition of Weil-\'etale cohomology in
characteristic
$0$.

This is a write-up of the talk I gave at the summer school on July 27, 2006. Much of the oral
version was tentative because the suitable $l$-adic realisation functors were not constructed
at the time. The final version is much more substantial than I had envisioned: this is both
because of technical difficulties and because I tried to make the exposition as pedagogical as
possible, in the spirit of the summer school. I hope the reader will bear with the first
reason, and be satisfied with the second one.

I also hope that some readers will, like me, find the coherence and beauty of the picture
below compelling reasons to believe in these conjectures.

I wish to thank Joseph Ayoub for a great number of exchanges while preparing this work, and the referee for a thorough reading which helped me improve the exposition.

\subsection*{Notation} $k$ denotes a perfect field; we write $\Sm(k)$ for
the category of smooth separated
$k$-schemes of finite type. When $k$ is finite we write $\F$
instead of $k$ and denote by $G\simeq \hat{\Z}$ its absolute Galois group. 

If $\sC$ is a category, we write $\sC(X,Y)$, $\Hom_\sC(X,Y)$ or $\Hom(X,Y)$ for the set of
morphisms between two objects $X,Y$, according to notational convenience.

\section{General overview}

This section gives a background to the sequel of the paper.

\subsection{Triangulated categories of motives} \label{0.1} As explained in Andr\'e's book
\cite[Ch. 7]{andre}, the classical conjectures of Hodge and Tate, and less classical ones of
Grothendieck and Ogus, may be interpreted as requesting certain realisation functors on pure
Grothendieck motives to be fully faithful. These conjectures concern algebraic cycles on smooth
projective varieties modulo homological equivalence. On the other hand, both Bloch's answer to
Mumford's nonrepresentability theorem for $0$-cycles
\cite[Lect. 1]{bloch} and Beilinson's approach to special values of $L$-functions
\cite{beiL,beilinson} led to conjectures on cycles modulo
\emph{rational} equivalence: the conjectures of Bloch-Beilinson and Murre (see Jannsen 
\cite{jannsen2} for an exposition). 

This development came parallel to another idea of Beilinson: in order to construct the
(still conjectural) abelian category $\sM(k)$ of mixed motives over a field $k$, one might start
with the easier problem of constructing a triangulated category of motives, leaving for later
the issue of finding a good $t$-structure on this category. Perhaps Beilinson had two main
insights: first, the theory of perverse sheaves he had been developing with Bernstein, Deligne
and Gabber \cite{bbd} and second, his vanishing conjecture for Adams eigenspaces on algebraic
$K$-groups (found independently by Soul\'e) which deals with an \emph{a priori} obstruction to the
existence of $\sM(k)$.

The latter programme: constructing triangulated categories of motives, was successfully
developed by Levine
\cite{levine}, Hanamura
\cite{hanamura} and Voevodsky \cite{V} independently. All three defined tensor triangulated categories
of motives over
$k$, by approaches similar in flavour but quite different in detail. It is now known that all
these categories are equivalent, if $\car k=0$ or if we take rational
coefficients.\footnote{Gabber's recent refinement of de
Jong's alteration theorem \cite{illusie-gabber} now allows us to just invert the exponential characteristic
for these theorems.} More precisely, the comparison between Levine's and Voevodsky's categories
is due to Levine in characteristic $0$  \cite[Part I, Ch. VI, 2.5.5]{levine} and to Ivorra in
general \cite{ivorra}, while the comparison between Hanamura's and Voevodsky's categories is
due to Bondarko \cite{bondarko} and independently to Hanamura (unpublished).

These three constructions extend when replacing the field $k$ by a rather
general base $S$ \cite{levine,voebase,hana11}\footnote{As far as I know, no comparison between these extensions has been attempted yet.}. At this stage, the issue
of Grothendieck's six operations \cite[5.10 A]{beilinson} starts to make sense. In a talk at the ICTP in 2002, Voevodsky
gave hints on how to carry this over in an abstract framework which would fit with his 
constructions, at least for the
four functors $f^*$, $f_*$,
$f_!$ and
$f^!$. This programme was taken up by Ayoub in \cite{ayoub}; he added a great deal to
Voevodsky's outline, namely a study of the missing operations $\otimes$ and $\uHom$ plus
related issues like constructibility and Verdier duality, as well as an impressive theory of
specialisation systems, a vast generalisation of the theory of nearby cycle functors.

It remained to see whether this abstract framework applied to categories of
motives over a base, for example to the Voevodsky version $S\mapsto \DM(S)$  constructed using
relative cycles (``sheaves with transfers"). It did apply to a variant ``without transfers"
$S\mapsto \DA(S)$ (as well as to Voevodsky's motivating example: the Morel-Voevodsky stable
$\A^1$-homotopy categories $S\mapsto \SH_{\A^1}(S)$): see \cite[Ch. 4]{ayoub} for this. It did
not apply directly to $\DM$, however. This issue was solved to some extent by Cisinski and
D\'eglise \cite{cis-deg}, who showed that the natural functor $\DA_\et(S,\Q)\to \DM(S,\Q)$
is an equivalence of categories when $S$ is a normal scheme, where $\DA_\et$ is an \'etale
variant of
$\DA$. All this will be explained in much more detail in \S \ref{s.base}.

\subsection{Motivic conjectures and categories of motives} It is both a conceptual and a
tactical issue to reformulate the conjectures alluded to at the beginning of \S \ref{0.1} 
 in this triangulated framework. The first necessary thing is to have
triangulated realisation functors at hand. In Levine's framework, many of them are
constructed in his book
\cite[Part I, Ch. V]{levine}. In Voevodsky's framework, with rational coefficients
and over a field of characteristic $0$, this was done by Huber using her triangulated category
of mixed realisations as a target
\cite{huber}. Over a separated Noetherian base, with integral coefficients and for $l$-adic
cohomology, this was done by Ivorra \cite{ivorrareal}.

Then came up the issue whether realisation functors commute with the six
operations. The only context where the question made full sense was Ivorra's. But there were
three problems at the outset: Ivorra's functors 1) are only defined on geometric motives, and
2) are contravariant. The third problem is that the formalism of six operations is not known
to exist on $S\mapsto \DM(S)$ in full generality, as explained above.

These issues were recently solved by Ayoub who constructed covariant $l$-adic realisation
functors from $\DA_\et(S)$ to Ekedahl's $l$-adic categories
$\hat{D}(S,\Z_l)$ \cite{ayoubreal}. He proved that they commute with the six operations and with
the right choice of a specialisation system. More details are in \S \ref{s.base}.

\enlargethispage*{30pt}

\section{The Tate conjecture: a review}

In this section, $\F=\F_q$ is a finite field with $q$ elements. The main reference here is
Tate's survey 
\cite{tate}.

\subsection{The zeta function and the Weil conjectures} Let $X$ be a
$\F$-scheme of finite type. It has a zeta function:
\[\zeta(X,s) = \exp\left(\sum_{n\ge 1} |X(\F_{q^n})|\frac{q^{-ns}}{n} \right)= \prod_{x\in
X_{(0)}} (1-|\F(x)|^{-s})^{-1}
\]

Weil conjectured that $\zeta(X,s)\in \Q(q^{-s})$ for any $X$: Dwork was first to
prove it in  \cite{dwork}. A different proof, based on
$l$-adic cohomology, was given by Grothendieck et al in \cite{sga5}. It provided the following
extra property, also conjectured by Weil:  if
$X$ is smooth projective, there is a functional equation of the form
\[\zeta(X,s) = A B^s \zeta(X,\dim X-s)\]
where $A,B$ are constants.

Continue to assume $X$ smooth projective and let $n\ge 0$. We write $\sZ^n(X)$ for the
group of cycles of codimension $n$ on $X$ and $A^n_\num(X)$ for its quotient by numerical equivalence:  $A^n_\num(X)$ is finitely
generated (see
\cite[3.4.6]{andre} or 
\cite[Th. 2.15]{riou}). The starting point of this work is the following

\begin{conj}[Tate \protect{\cite{tatepoles}}]\label{ct} $\ord_{s=n}\zeta(X,s) = - \rk
A^n_\num(X)$.
\end{conj}

\subsection{Cohomological interpretation} The proofs of \cite{sga5} give the factorisation
(another Weil conjecture)
\[\zeta(X,s) = \prod_{i=0}^{2d} P_i(q^{-s})^{(-1)^{i+1}}\]
with
\begin{gather*}
P_i(t) = \det(1-\phi t\mid H^i_l(X))\\
 \phi=\text{Frobenius},\quad H^i_l(X):= H^i_\cont(\bar X, \Q_l)\quad (l\ne p).
\end{gather*}

Here $\bar X=X\otimes_\F \bar \F$, $H^i_\cont(\bar X, \Q_l):= \lim
H^i_\et(\bar X,\Z/l^n)
\otimes
\Q_l$. 

Finally the last Weil conjecture
was proven by Deligne in
\cite{weilI}: 

\begin{thm}[``Riemann Hypothesis over finite fields"]\label{rh} For all $i$, $P_i(t)\in \Z[t]$ and its inverse roots have complex absolute
values $= q^{i/2}$.
\end{thm}

This yields:

\begin{thm}[Milne \protect{\cite[Prop. 8.2 and 8.4]{milneamer}}, Tate \protect{\cite[Th. (2.9)]{tate}}]\label{tt} For any $(X,n,l)$, the
following are equivalent:
\begin{enumerate}
\item Conjecture \ref{ct};
\item $\rk A^n_\num(X)=\dim_{\Q_l} H^{2n}_\cont(\bar X, \Q_l(n))^G$, $G=
Gal(\bar \F/\F)$;
\item 
\begin{enumerate}
\item the cycle map $\sZ^n(X)\otimes \Q_l\to H^{2n}_\cont(\bar X, \Q_l(n))^G$ is
surjective,
\item homological and numerical equivalences agree on $\sZ^n(X)\otimes\Q$.
\end{enumerate}
\end{enumerate}
\end{thm}

(3) (a) is called the \emph{cohomological Tate conjecture}: it makes sense (and is conjectured
to hold) over any finitely generated field.

\begin{remark}\label{sn} Let us write $S^n$ for the following condition: the composition
\[H^{2n}_\cont(\bar X, \Q_l(n))^G\Inj H^{2n}_\cont(\bar X,
\Q_l(n))\Surj H^{2n}_\cont(\bar X, \Q_l(n))_G\]
is bijective. Poincar\'e duality easily implies that $S^n\iff S^{d-n}$ if $d=\dim X$. Let us
also write more precisely (a)$_n$ for (a) in Theorem \ref{tt} (3). Then, by
\cite[Th. (2.9)]{tate}, the conditions of this theorem are also equivalent to
\begin{itemize}
\item[(4)] (a)$_n$ + (a)$_{d-n}$ + $S^n$.
\end{itemize}
Condition $S^n$ is verified if $X$ is an abelian variety,  essentially because the Frobenius
endomorphism of $X$ is semi-simple as an element of the centre of $\End^0(X)$ \cite[lemme
1.9]{cell}. Thus, in this case, (1) $\iff$ (a)$_n$ + (a)$_{d-n}$ in Theorem \ref{tt}.
\end{remark}

\begin{remark}\label{r2} Here are other known consequences of Conjecture \ref{ct} (see Tate
\cite{tate} and Milne \cite{milne}):
\begin{thlist}
\item Grothendieck's standard conjecture B (by Theorem \ref{tt} (3) (b)).
\item Semi-simplicity of Galois action on $l$-adic cohomology (\cite[Rk. 8.6]{milneamer}, 
\cite{ss}).
\item Any (pure) homological motive $M$ is a direct summand of $h(A_K)(n)$ for an abelian $\F$-variety $A$,
$K/\F$ a finite extension and $n\in\Z$ \cite[Rk. 2.7]{milne}. We say that $M$ is \emph{of abelian type}.
\item ``Any mixed motive [over a finite field!] is pure" \cite[Th. 2.49]{milne} (see also Proposition \ref{p10.2}
below).
\end{thlist}
\end{remark}

\subsection{Known cases of the cohomological Tate conjecture}\label{known} They are scarce: see
\cite[\S 5]{tate} for a 1994 state of the art. 

The main case is $n=1$, where the conjecture is known for  abelian varieties over any finitely
generated field $k$ (Tate \cite{tatefinite} if $k$ is finite, Zarhin \cite{zarhin} if $\car k>
0$, Faltings \cite{faltings,faltings-wusth} if $\car k=0$). This trivially extends to smooth
projective varieties $X$ whose homological motive is of abelian type in the sense of Remark
\ref{r2} (iii): for simplicity, we shall then say that $X$ is \emph{homologically of abelian
type}. For some other $X$'s, see \cite[Th. (5.6) and (5.8)]{tate}. 

One can then sometimes extend the cohomological Tate conjecture from codimension $1$ to all
codimensions. It was essentially observed by Soul\'e that this is automatic if $X$ is
homologically of abelian type and of dimension $\le 3$ \cite[Th. 3]{soule}.\footnote{Soul\'e's
proof works over a finite field, but the argument of \cite[proof of Th. 82]{handbook}
extends this to any finitely generated field.} In another direction, this will work if the
algebra $H^{2*}_\cont(X\otimes_K \bar k,\Q_l(*))^G$ is generated in cohomological degree $2$:
this happens when $k$ is finite and $X$ is a product of elliptic curves (Spiess \cite{spiess}),
and also for powers of certain simple abelian varieties (Lenstra, Zarhin). See \cite{milne2},
where Milne also proves the Tate conjecture for powers of certain abelian varieties for which
the above generation condition is not satisfied.

\section{The Beilinson, Parshin and Friedlander conjectures}

\subsection{The conjectures}

\begin{conj}[Beilinson]\label{cb} Let $X/\F$ be smooth projective  and let
$n\ge 0$. Then rational and numerical equivalences coincide on
$\sZ^n(X)\otimes\Q$.
\end{conj}

This is the strongest possible conjecture: it implies that all adequate equivalence relations
should coincide on
$\sZ^n(X)\otimes\Q$! (Compare \cite[3.2.2.1 and 3.2.7.2]{andre}.)

\begin{conj}[Parshin]\label{cp} For any smooth projective $\F$-variety $X$, the algebraic
$K$-group $K_i(X)$ is torsion for $i>0$.
\end{conj}

This contains in particular the Beilinson-Soul\'e vanishing conjecture on eigenspaces for the
Adams operations (\cite{beiL}, \cite[p. 501, Conj.]{soule2}).

\begin{remark}  Milne writes in \cite[\S 1]{milneaim} that Tate already formulated conjecture
\ref{cb} orally in the Woods Hole seminar of 1964. On the other hand, Conjecture \ref{cp} is
attributed to Parshin by Jannsen in \cite[Conj. 12.2]{jannsenhab}, but is also formulated by
Beilinson in \cite[Conj. 8.3.3 b)]{beiAH}\footnote{Beilinson attributes it to Parshin in
\cite[Conj. 2.4.2.3]{beiL}.}. So the attributions of these conjectures are a bit blurred. For
simplicity, I will keep their established names of ``Tate conjecture", ``Beilinson conjecture"
and ``Parshin conjecture" here.
\end{remark}

\subsection{``Elementary" implications}
We now start reviewing implications between these conjectures, in the same spirit as
Theorem \ref{tt} and Remark \ref{r2}.

\begin{thm}[Geisser \protect{\cite{geisser}}]\label{tg} Conjecture \ref{ct} +
Conjecture
\ref{cb}
$\Rightarrow$ Conjecture \ref{cp}.
\end{thm}

See Theorem \ref{tG} below for a more precise version.

Geisser's proof is based on Jannsen's semi-simplicity theorem \cite{jannsen} plus a Frobenius
argument going back to Soul\'e \cite{soule}. The same is true for the proof of the following
theorem, in which $CH^d$ is the quotient of $\sZ^d$ by rational equivalence.

\begin{thm}[Kahn \protect{\cite{cell}}]\label{tk} Let $X$ be a smooth projective $\F$-variety
of dimension
$d$. Assume that $\Ker(CH^d(X\times X)\otimes \Q\to A^d_\num(X\times X)\otimes \Q$ is a
nilpotent ideal. Then Conjecture \ref{ct} for $X$ and all $n\ge 0$ $\Rightarrow$ Conjecture
\ref{cb} for
$X$ and all $n\ge 0$.
\end{thm}

The point is that the hypothesis of Theorem \ref{tk} is
verified if the Chow motive of
$X$ is
\emph{finite dimensional} in the sense of Kimura and O'Sullivan (see \cite[Ch. 12]{andre} or
\cite{ivorraIHES}). This will be the case if $X$ is \emph{of abelian type}, that is, if its
motive verifies condition (iii) of Remark
\ref{r2} for \emph{rational equivalence} (\emph{ibid.}). This gives many cases where one can
apply Theorem
\ref{tk}, see \S \ref{known}.

\subsection{From smooth projective to more general varieties}\label{s2.3} Putting together the
cohomological Tate conjecture ((3) (a) in Theorem \ref{tt}) and the Beilinson conjecture
(Conjecture \ref{cb}) gives the following

\begin{conj}\label{cf} For any smooth projective $\F$-variety $X$ and any $n\ge 0$, the cycle
class map
\[CH^n(X)\otimes \Q_l\to H^{2n}(\bar X,\Q_l(n))^G\]
is bijective.
\end{conj}

This statement continues to make sense when removing the word ``projective". It was 
conjectured to be true by Friedlander. Actually, Friedlander made
a general conjecture involving the $l$-adic Chern character from algebraic $K$-theory, with
targets of the form $H^i(\bar X,\Q_l(n))^G$ for all $i\in \Z$ (ibid.), namely:

\begin{conj}[Friedlander, \protect{\cite[8.3.4 b)]{beiAH}}]\label{cff} For any smooth
$\F$-variety $X$, the Chern character
\[ch_{n,i}:K_{2n-i}(X)_\Q^{(n)}\otimes \Q_l\to H^i(\bar X,\Q_l(n))^G\]
is an isomorphism, where the left hand side is the $n$-th Adams eigenspace on the algebraic
$K$-theory of $X$.
\end{conj}

For $i=2n$ we get the Adams eigenspace of weight
$n$ on $K_0(X)\otimes\Q$
which may be identified with $CH^n(X)\otimes
\Q$ thanks to the Grothendieck Riemann-Roch theorem, so Friedlander's
conjecture includes Conjecture \ref{cf} in the extended case of smooth varieties. For
$X$ smooth projective, it implies Parshin's conjecture
\ref{cp} in view of Theorem \ref{rh}, as observed by Beilinson
(ibid.). Geisser's theorem \ref{tg} refines this observation.

\begin{remarks} \
\begin{enumerate}
\item I want to avoid entering the details of algebraic $K$-theory
and the construction of $l$-adic Chern characters here, in order to concentrate on the already
substantial case of motivic cohomology and $l$-adic cycle class maps. 
\item Chern characters
involve denominators, so that Friedlander's conjecture is intrinsically with rational
coefficients. One could get some integrality by using higher Chern classes, but not completely
avoid the denominators. On the other hand, as we shall see, a motivic version of this
conjecture can be formulated integrally.
\end{enumerate}
\end{remarks}

In his habilitation thesis \cite[Th. 12.7 a)]{jannsenhab}, Jannsen proved that Friedlander's
conjecture follows from its restriction to smooth projective varieties, provided one adds a
generalisation of Condition
$S^n$ from Remark \ref{sn} and assumes resolution of singularities. Jannsen's proof is a
little awkward mainly because the formulation of Friedlander's conjecture involves fixed points
under Galois action, which is not an exact functor: this is where the recourse to
semi-simplicity (and the weight filtration) was essential.

Jannsen also gives an extension of these conjectures to singular varieties using $K'$-theory
and (Borel-Moore) $l$-adic homology.

Let $K$ be a finitely generated field of characteristic $p$, and let $X$ be a smooth
$K$-variety. Then $X$ spreads as a smooth $S$-scheme of finite type $\sX$ for a suitable smooth model $S$ of
$K$ over $\F_p$. Using this remark, one sees that Friedlander's conjecture potentially implies
statements concerning $K$-varieties: this is one of its many interests. For example, one
easily sees that it implies Conjecture 2.4.2.2 of \cite{beiL} on rational $K$-groups of fields
of positive characteristic, see \cite[8.28 (v) and 8.30]{tatesheaf} and \cite[Th. 60 and 61]{handbook}.


\subsection{Motivic reformulations} Since the appearance of \cite{jannsenhab}, two major
progresses have been made: 1) de Jong's theorem on alterations \cite{dJ}, and 2) definitions of
motivic cohomology, together with motivic $l$-adic cycle class maps. 

Motivic cohomology groups generalise Chow groups, and they agree
rationally with Adams eigen\-spaces on algebraic $K$-theory. For smooth varieties
there are mainly two versions of these groups: Bloch's higher Chow groups \cite{blochhigher} and
the Suslin-Voevodsky motivic cohomology groups. They are shown to be canonically (but not very
directly) isomorphic over any field in \cite{allagree}. 

In \cite{tatesheaf}, I was interested in formulating a version of Conjecture
\ref{cff} using motivic cohomology. I used the Suslin-Voevodsky version, but assumed resolution
of singularities (over finite fields!) because it was built in Voevodsky's theory of
triangulated motives at the time. To get around resolution, I changed course and used higher
Chow groups in
\cite{glr} and
\cite{cell}. Here I want to give a version of
Friedlander's conjecture \ref{cff} involving realisation functors on Voevodsky's $\DM(\F)$, thus
going back to the original viewpoint of using Suslin-Voevodsky motivic cohomology. It has
several advantages:

\begin{enumerate}
\item While Bloch's theory avoids resolution of singularities, it relies on a subtle
sophistication of Chow's moving lemma which renders the functoriality of his cycles complexes
(as opposed to their homology groups) delicate. By contrast, the motivic complexes of Suslin
and Voevodsky have a straightforward functoriality.
\item Much of the resolution of singularities needed in the basics of Voevodsky's theory has
been removed. On the one hand, he proved very simply the cancellation theorem over any perfect
field \cite{voecan}. On the other hand, de Jong's alteration theorem turns out to be sufficient
for most of the development, even integrally.
\item Some of the remaining original constructions of Voevodsky, which still require resolution
of singularities, like the motives (with or without compact supports) of singular schemes, can
now be bypassed by using the $4$ operations., see \S \ref{s.sing}.
\item The triangulated framework is much richer than the one involving smooth varieties: for
example, it automatically yields conjectural statements for the cohomology and cohomology with
compact supports of possibly singular varieties in the spirit of Jannsen, see (3).
\end{enumerate} 

\enlargethispage*{20pt}

\subsection{A reading guide} Let me now outline the sequel of this paper. I tried to go step
by step and by increasing degree of sophistication:

\begin{enumerate}
\item In \S \ref{s.motl}, I introduce motivic cohomology \`a la Suslin-Voevodsky by using
their concrete motivic complexes; the construction of the $l$-adic cycle class map is then very
easy but insufficient.
\item In \S \ref{s.DM}, I introduce Voevodsky's categories of motives over a field; the
motivic cohomology of \S \ref{s.motl} may then be interpreted as Hom groups in these categories
(Theorem \ref{t3.1} and \ref{t3.2}).
\item In \S \ref{s.base}, I describe the situation over a base, and in particular outline the
formalism of six operations. Here the story becomes more complicated as several versions of
categories of motives come into play. The main point is that they often agree: see \S
\ref{s.recap}. An important outcome is the construction of motives (with or without compact
supports) of singular varieties, with the right functoriality: see \S \ref{s.sing}. The
$l$-adic realisation functor is finally introduced in \S \ref{s.real}.
\item \S \ref{s.efsp} is technical but essential to get Theorem \ref{tmain} of the
introduction.
\item In \S \ref{s.geisser2}, I prove a more precise version of Geisser's Theorem \ref{tg}.
The main reason why I did this was to convince myself that there is no way to get the equivalence (ii) $\iff$ (iv) of Corollary \ref{c9.1} ``$n$ by $n$". 
\item In \S \ref{sFried}, we get slowly towards the motivic formulation of Friedlander's
conjecture and the proof of Theorem \ref{tmain}: this is achieved in \S \ref{sfried}, see 
Proposition \ref{rk=f}, Theorem \ref{t1} and Corollary \ref{c9.1}. In particular, the conditions
of Theorem \ref{tmain} are equivalent to Friedlander's conjecture plus a generalisation of
Condition $S^n$ of Remark \ref{sn}.  A simple lemma of Ayoub (Lemma
\ref{l9.2}) drastically simplifies an argument I had outlined in my 2006 talk.
\item In \S \ref{s.fg}, I give two other reformulations of these conjectures: a) a realisation
functor from a suitable category is fully faithful, b) the Hom groups of this category are
finitely generated. See  Proposition \ref{p10.1} and Theorem \ref{t10.1}. This involves an
apparently artificial construction, which is better explained in the next section. 
\item In \S \ref{s.w},  I recall Lichtenbaum's Weil-\'etale topology and inflict it to $\DM$.
This sheds a better light on the previous section, but these new categories have their
shortcomings.
\end{enumerate}


\section{Motivic cohomology and $l$-adic cohomology}\label{s.motl}

\subsection{Motivic cohomology, Nisnevich and \'etale}\label{s.motcoh} Let us first recall
Voevodsky's groups of \emph{finite correspondences} \cite{V}: if $X,Y$ are smooth $k$-schemes,
we denote by
$c(X,Y)$ the free abelian group with basis the closed integral subschemes $Z\subseteq X\times
Y$ such that the projection $Z\to X$ is finite and surjective on some connected component of
$X$. The beauty of finite correspondences is that they compose ``on the nose", that is, without
having to mod out by any adequate equivalence relation. In particular, they are covariant in $Y$
and contravariant in $X$.

If $(Y_1,y_1),(Y_2,y_2)$ are two pointed smooth $k$-schemes ($y_i\in Y_i(k)$), we define
\begin{multline*}
c(X,(Y_1,y_1)\wedge(Y_2,y_2)) =\\ \Coker\left(c(X,Y_1\times y_2)\oplus
c(X,y_1\times Y_2)\to c(X,Y_1\times Y_2)\right)
\end{multline*}
and more generally, we define $c(X,(Y_1,y_1)\wedge\dots\wedge (Y_n,y_n))$ by a similar formula:
this is a direct summand of $c(X,Y_1\times\dots\times Y_n)$.

We denote by $\Delta^\cdot$ be the standard cosimplicial $k$-scheme, with 
\[\Delta^n= \Spec k[t_0,\dots,t_n]/(\sum t_i-1).\]

Let now $X\in \Sm(k)$, and let $n\ge 0$. The $n$-th \emph{Suslin-Voevodsky complex} is
the complex whose $i$-th term is
\[C^i(X,\Z(n)) = c(X\times \Delta^{n-i},\G_m^{\wedge n})\]
where $\G_m^{\wedge n}=(\G_m,1)\wedge\dots\wedge (\G_m,1)$, and the differentials are as usual.

Varying $X$, this defines a complex of presheaves (even sheaves) over the \'etale site on
$\Sm(k)$, and \emph{a fortiori} over the corresponding Nisnevich
site. Let us denote this complex by $\underline{C}^*(\Z(n))$. We set

\begin{defn}\label{d3.1} a) $H^i(X,\Z(n))= \bH^i_\Nis(X,\underline{C}^*(\Z(n)))$. This is
\emph{motivic cohomology}.\\ 
b) $H^i_\et(X,\Z(n))= \bH^i_\et(X,\underline{C}^*(\Z(n)))$. This is \emph{\'etale
motivic cohomology}.
\end{defn}

The link between motivic cohomology and Chow groups is the following:

\begin{prop} One has isomorphisms
\[CH^n(X)\simeq H^{2n}(X,\Z(n)).\]
\end{prop}

More generally, motivic cohomology groups (for the Nisnevich topology) are isomorphic to Bloch's
higher Chow groups, with a suitable reindexing \cite{allagree}. This also implies:
\begin{equation}\label{eq4.3}
H^i(X,\Z(n))=0 \text{ for } i>2n.
\end{equation}

\subsection{Continuous \'etale cohomology} We now assume that $k=\F$ is
finite until the end of this section.

Let $X\in\Sm(\F)$. We write $H^i_\cont(X,\Z_l(n))$ for Jann\-sen's \emph{continuous \'etale
cohomology} \cite{jannsencont}: recall that if $(A_n)$ is an inverse system of \'etale sheaves, $H^i_\cont(X,(A_n))$ is defined as the $i$-th derived functor of the composite functor
\[(A_n)\mapsto (\Gamma(X,A_n))\mapsto \lim \Gamma(X,A_n).\]

We also write $H^i_\cont(X,\Q_l(n)):=H^i_\cont(X,\Z_l(n))\otimes\Q$. 

Since \'etale cohomology of $X$ with finite coefficients is finite, this coincides with the
na\"\i ve $l$-adic cohomology:
\begin{equation}\label{eq3.0}
H^i_\cont(X,\Z_l(n))\iso \lim H^i_\et(X,\mu_{l^\nu}^{\otimes n}).
\end{equation}

The Hochschild-Serre spectral sequence yields short exact sequences:
\begin{equation}\label{eq3.4}
0\to H^{i-1}_\cont(\bar X,\Z_l(n))_G\to H^i_\cont(X,\Z_l(n))\to H^i_\cont(\bar X,\Z_l(n))^G\to
0.
\end{equation}

\subsection{An $l$-adic lemma}  
Let
\begin{equation}\label{eq6.1c}
e\in H^1_\cont(\F,\hat{\Z})\simeq \Hom_\cont(\hat{\Z},\hat{\Z})\simeq \hat{\Z}
\end{equation}
be the canonical generator corresponding to the identity, where the first isomorphism is given
by the arithmetic Frobenius of $\F$.  Note that $e^2=0$  since $cd(\F)=1$. The
following generalises  \cite[Prop. 6.5]{milneamer} (the proof is the same as for \cite[Prop. 6.5]{tatesheaf}):

\begin{lemma}\label{l4.2} Let $C\in \hat{D}_c^b(\F,\Z_l)$ be a perfect complex of $l$-adic
sheaves over
$\Spec \F$. Then the diagram
\begin{equation}\label{eq9.2}
\xymatrix{
H^j_\cont(\F,C) \ar[r]^a \ar[dd]_{\cdot c} & H^j(C)^G\ar[dr]^{i}\\
 && H^j(C)\ar[dl]^p\\
H^{j+1}_\cont(\F,C)& H^j(C)_G\ar[l]_b
}
\end{equation}
commutes for any $j\in\Z$. Here $i$ is the inclusion, $p$ is the projection and $a,b$ are
the maps stemming from the hypercohomology spectral sequence for $H^*_\cont(\F,C)$.\qed
\end{lemma}

\subsection{Na\"\i ve cycle class maps}\label{s.naive} The rigidity theorem of Suslin-Voe\-vod\-sky
(\cite[Th. 7.6]{suvo}  or \cite[Th. 6.1]{VIHES}) yields canonical isomorphisms
\begin{equation}\label{eq3.3}
H^i_\et(X,\Z(n)\otimes \Z/l^\nu)\simeq H^i_\et(X,\mu_{l^\nu}^{\otimes n})
\end{equation} 
which, in view of \eqref{eq3.0}, immediately
yield  higher $l$-adic cycle class maps:
\begin{equation}\label{eq3.2}
H^j_\et(X,\Z(n))\otimes \Z_l \to H^{j}_\cont(X,\Z_l(n)).
\end{equation}

One has isomorphisms \cite{blochhigher}
\[K_{2n-i}(X)_\Q^{(n)}\simeq H^j(X,\Q(n))\iso H^j_\et(X,\Q(n)),\]
which convert the maps \eqref{eq3.2} into the Chern characters $ch_{n,i}$ of Conjecture
\ref{cff}. Thus, this construction already allows us to reformulate Conjecture
\ref{cff} in terms of motivic cohomology. However it is too crude to provide
a correct \emph{integral} version of this conjecture (see \S \ref{s6.1} for a detailed
discussion). For this, one needs a triangulated version of
\eqref{eq3.2}, namely a map
\[\Z(n)_\et\otimes \Z_l\to R\lim \mu_{l^\nu}^{\otimes n}\]
in a suitable triangulated category, which has to be modified. In \cite{glr}, this was done in
the derived category of
\'etale sheaves on $Sm(\F)$, using Bloch's cycle complexes instead of the Suslin-Voevodsky
complexes. Since we are going to ultimately give a reformulation using $\DM_\et^\eff(\F)$, it is
best to construct the above morphism directly in this category. This will be done in \S
\ref{s.rel}.

\section{Categories of motives over a field}\label{s.DM} 

\subsection{Categories of effective geometric motives}\label{s.motgm} Let $\DM_\gm^\eff(k)$
be Voevodsky's tensor triangulated category of effective geometric motives \cite{V}: we also
refer to Levine's survey \cite[Lect. 1 and 2]{levineIHES} for a nice exposition. Then any
$X\in \Sm(k)$ has a
\emph{motive}
$M(X)\in\DM_\gm^\eff(k)$. The functor $X\mapsto M(X)$ has the following properties:

\begin{description}
\item[Homotopy invariance] $M(X\times \A^1)\iso M(X)$.
\item[Mayer Vietoris] If $X = U\cup V$ with $U,V$ open, then one has an exact triangle
\[M(U\cap V)\to M(U)\oplus M(V)\to M(X)\by{+1}.\]
\item[Projective bundle formula] Write $\Z(1)$ for the direct summand of $M(\P^1)[-2]$ defined
by any rational point: this is the \emph{Tate object}. For $n>0$, define $\Z(n)=\Z(1)^{\otimes
n}$: if $M\in \DM_\gm^\eff(k)$, we write $M(n):=M\otimes \Z(n)$. Now let $E\to X$ be a vector
bundle of rank
$r$. Then
\[M(\P(E))\simeq \bigoplus_{n=0}^{r-1} M(X)(n)[2n].\]
\item[Gysin] Let $Z\subset X$ be a closed subscheme, smooth of codimension $c$, and let
$U=X-Z$. Then one has an exact triangle
\[M(U)\to M(X)\to M(Z)(c)[2c]\by{+1}.\]
\end{description}

Let us recall the construction of $\DM_\gm^\eff(k)$ in two words: it parallels the construction
of Chow motives to a large extent, but is based on smooth, not necessarily projective,
varieties. We start from the category
$\Cor(k)$ of
\emph{finite corrrespondences} over $k$: its objects are smooth $k$-schemes of finite type and
its morphisms are the groups $c(X,Y)$ of \S \ref{s.motcoh}. The graph of a morphism is a finite
correspondence, which yields a (covariant) functor $\Sm(k)\to \Cor(k)$.

Since $\Cor(k)$ is an additive category, one may consider the homotopy category $K^b(\Cor(k))$
of bounded chain complexes of objects of $\Cor(k)$. Then $\DM_\gm^\eff(k)$ is defined as the
pseudo-abelian envelope of the Verdier localisation of $K^b(\Cor(k))$ with respect to
the ``relations" corresponding to homotopy invariance and Mayer-Vietoris as above. This yields
a string of functors:
\[\Sm(k)\to \Cor(k)\to K^b(\Cor(k))\to \DM_\gm^\eff(k).\]

Their composition is the above functor $M$.

\subsection{Categories of motivic complexes, Nisnevich topology}\label{s.moteff}
The relationship between motivic cohomology and $\DM_\gm^\eff(k)$ is

\begin{thm}[\protect{\cite{V}}]\label{t3.1} One has canonical isomorphisms
\[H^i(X,\Z(n))\simeq \Hom_{\DM_\gm^\eff(k)}(M(X),\Z(n)[i]).\]
\end{thm}

This is proven by introducing a larger tensor triangulated category $\DM_-^\eff(k)$,
constructing a fully faithful functor $\DM_\gm^\eff(k)\to \DM_-^\eff(k)$ and proving Theorem
\ref{t3.1} in
$\DM_-^\eff(k)$.

In short, $\DM_-^\eff(k)$ is constructed out of Nisnevich sheaves with transfers. A
\emph{presheaf with transfers} $\sF$ is simply an additive contravariant functor from the
category $\Cor(k)$ of the previous subsection to abelian groups; $\sF$ is a \emph{Nisnevich
sheaf with transfers} if, viewed as a functor on $\Sm(k)$, it is a sheaf for the Nisnevich
topology. Write $\NST(k)$ for the abelian category of Nisnevich sheaves with transfers. Inside
the bounded above derived category $D^-(\NST(k))$, Voevodsky considers the full subcategory
consisting of those complexes whose homology sheaves are \emph{homotopy invariant}: this is
$\DM_-^\eff(k)$. The embedding of $\DM_\gm^\eff(k)$ into $\DM_-^\eff(k)$ may then be thought of
as a Yoneda embedding: it is constructed using complexes $\underline{C}_*(\sF)$ generalising
those described in \S \ref{s.motcoh}.

The category $\DM_-^\eff(k)$ is constructed out of bounded above complexes of Nisnevich
sheaves. It will be more convenient for us to work with the corresponding
category $\DM^\eff(k)$ constructed in the same way but without boundedness conditions. In fact
we have a commutative square of inclusions of tensor triangulated categories
\begin{equation}
\begin{CD}\label{eq4.2}
\DM_\gm^\eff(k) @>a>> \DM^\eff(k)\\
@VcVV @VdVV\\
\DM_\gm(k) @>b>> \DM(k).
\end{CD}
\end{equation}

The two bottom categories are obtained from the top categories by inverting the Tate object,
but in rather different ways. On the left hand side, one defines $\DM_\gm(k)$ as the category
whose objects are pairs $(M,m)$ with $M\in \DM_\gm^\eff(k)$ and $m\in\Z$, and Hom sets are
defined by
\[\Hom_{\DM_\gm(k)}((M,m),(N,n))=\colim_{r\ge -m,-n} \Hom_{\DM_\gm^\eff(k)}(M(m+r),N(n+r))\]
cf. \cite[p. 192]{V}. The full faithfulness of $a$ follows from [the proof of] \cite[Th.
3.2.6]{V} and the full faithfulness of $c$ follows from \cite{voecan}; both results use the
perfectness of the field $k$.

On the other hand, if we try and define $\DM(k)$ from $\DM^\eff(k)$ in the same way, the
category we get is too small: in particular it does not have infinite direct sums, while
$\DM^\eff(k)$ does. The solution is to pass to the model-categoric level and imitate the
construction of spectra in algebraic topology: namely one considers ``$\Z(1)$-spectra" of the
form
\[(C_n,e_n)_{n\in\Z}\]
where $C_n\in C(NST)$ and $e_n$ is a morphism of complexes $C_n\otimes \Z(1)\to C_{n+1}$ for
$\Z(1)$ as in \S \ref{s.motcoh}. This is outlined by Morel in
\cite[\S 5.2]{morel}, as an analogue to the construction given by Voevodsky in \cite[\S
5]{VICM} for the stable homotopy category of schemes. A detailed construction, over a base, is
given by Cisinski and D\'eglise in \cite[Ex. 7.15]{cis-deg1}.

One can show that the functors $b$ and $d$ are also fully faithful. The categories
$\DM^\eff(k)$ and $\DM(k)$ are closed, i.e., have internal Homs.

\subsection{Variant} If $A$ is a commutative ring, one may define similar categories of
motives with coefficients in $A$; we shall denote them by $\DM(k,A)$ and the decorated
analogues. The most important cases for us will be $A=\Q$, $A=\Z/n$ and, later, $A=\Z_l$.

\subsection{Remark on notation} I will never write $\DM(k,\Z)$ for $\DM(k)$. There is a
dilemma (or a trilemma) for writing the unit object of $\DM(k,A)$. A spontaneous choice would
be $A$, but this leads to ambiguities. One ambiguity which would be disastrous here would be to
confuse notationally the complex of sheaves $\Z(n)\otimes \Z_l$ and the $l$-adic sheaf
$\Z_l(n)$.

Another choice which has been made elsewhere to avoid this problem is to use the neutral
notation $\un$ for the unit object. This leads to notation like $\un(n)[i]$, which I find
disgracious.

For these reasons, I have decided to retain the notation $\Z$ for the unit object of
$\DM(k,A)$ even when $A$ is another ring than $\Z$. Similarly, the notation $\Z(n)$ is retained
in $\DM(k,A)$ for any $A$. This means that $\DM(k,A)(\Z,\Z)\allowbreak =A$ and
$\DM(k,A)(M(X),\Z(n)[i]) = H^i(X,\Z(n))\otimes A$ for $X\in \Sm(k)$.

This also applies to the variants of $\DM(k)$ to come.

On the other hand, I will sometimes write $H^i(X,A(n))$ for $H^i(X,\Z(n))\otimes A$ when there is no risk of ambiguity, e.g. for $A=\Q$ or $\Z/m$ (but not $A=\Z_l$).

 I hope this will not create any confusion.
 
\subsection{Categories of motivic complexes, \'etale topology} Let $A$ be a commutative ring.
There is an
\'etale analogue to the inclusion
$\DM^\eff(k,A)\allowbreak\Inj
\DM(k,A)$, using the \'etale topology
rather than the Nisnevich topology. This yields a naturally commutative diagram of tenros
triangulated categories:
\begin{equation}\label{eq4.1}
\begin{CD}
\DM^\eff(k,A) @>\alpha^*>> \DM_\et^\eff(k,A)\\
@V{d}VV @V{d_\et}VV\\
\DM(k,A) @>\alpha^*>> \DM_\et(k,A)
\end{CD}
\end{equation}
where $\alpha^*$ denotes the change of topology functor. The functor $d_\et$ is still fully
faithful.

The basic results on the effective categories are identical
to those stated in \cite[Prop. 3.3.2 and 3.3.3]{V} for their bounded versions:

\begin{thm}\label{t3.2} Assume that the field $k$ has finite cohomological dimension, and
exponential characteristic $p$. Then:\\ 
a) For $A=\Q$, the functors $\alpha^*$ of \eqref{eq4.1} are
equivalences of tensor triangulated categories.\\
b) for $n>0$ prime to $p$, the natural functor
\[\DM^\eff_\et(k,\Z/n)\to D(k_\et,\Z/n)\]
is an equivalence of triangulated categories, where the right hand side is the derived category
of sheaves of $\Z/n$-modules on the small \'etale site of $\Spec k$ and the functor is induced
by restriction to the small \'etale site.\\
c) For $\nu \ge 1$, the categories $\DM^\eff_\et(k,\Z/p^\nu)$ and $\DM_\et(k,\Z/p^\nu)$ are
$0$. 
\end{thm}

Part c) of Theorem \ref{t3.2} means that $\DM_\et^\eff(k)$ and $\DM_\et(k)$ are $\Z[1/p]$-linear;
the basic reason why this follows from homotopy invariance is the Artin-Schreier exact sequence
(of
\'etale shaves with transfers)
\[0\to \Z/p\to \G_a\by{F-1} \G_a\to 0.\]

One deduces a pendant to Theorem \ref{t3.1}, plus an important comparison between motivic
and \'etale motivic cohomology:

\begin{thm}\label{t3.3} a) One has canonical isomorphisms
\[H^i_\et(X,\Z(n))\otimes \Z[1/p]\simeq \Hom_{\DM_\et^\eff(k)}(M(X),\Z(n)[i]).\]
b) For any smooth $X$ and any $n\ge 0$, $i\in\Z$, the natural map
\[H^i(X,\Z(n))\otimes \Q\to H^i_\et(X,\Z(n))\otimes \Q\]
is an isomorphism.
\end{thm}

An upshot is that if we work with $\DM$, we are bound to lose $p$-torsion
information on \'etale motivic cohomology. Using the groups from Definition \ref{d3.1} b)
retains this information, but we lose the power of the triangulated formalism.\footnote{To
recover this power for the $p$-torsion, one would have to invent new triangulated categories
which would accomodate non homotopy invariant phenomena: this is very much work in progress now.}

We shall also need:

\begin{lemma}\label{l4.1} Let $\DM_{\gm,\et}(k)$ be the full subcategory of $\DM_\et(k)$ which
is the pseudo-abelian envelope of the image of $\DM_\gm(k,\Z[1/p])\to \DM_\et(k)$. Then
$\DM_{\gm,\et}(k)$ is a rigid subcategory of $\DM_\et(k)$: every object is strongly dualisable.
\end{lemma}

(Recall that an object $M$ of a unital symmetric monoidal category $\sA$ is \emph{strongly
dualisable} if there exists a triple $(M^*,\eta,\epsilon)$ with $M^*\in \sA$ and
$\eta:\un\to M\otimes M^*$, $\epsilon:M^*\otimes M\to \un$, verifying adjunction-like
identities. Such a triple is then unique up to unique isomorphism; see \cite{dp}.)

\begin{proof} Since $\DM_\gm(k,\Z[1/p])\to \DM_\et(k)$ is a (strict) $\otimes$-functor, it
suffices to show that $\DM_\gm(k,\Z[1/p])$ is rigid. But this triangulated category is generated
by motives of smooth projective varieties, as follows from Gabber's refinement of de Jong's
theorem, and these objects are strongly dualisable (cf. \cite[App. B]{HK}).
\end{proof}

\subsection{Motivic complexes in $\DM$} They are easy to understand in weight
$\le 1$:

\begin{thm}[see \protect{\cite[Cor. 3.4.3]{V}}] One has isomorphisms in $\DM^\eff(k)$
\begin{align*}
\Z(0)&\simeq \Z\\
\Z(1)&\simeq \G_m[-1]
\end{align*}
where $\G_m$ is viewed as a homotopy invariant Nisnevich sheaf with transfers.
\end{thm}

These quasi-isomorphisms persist when passing to $\DM_\et^\eff(k)$: we shall put an index
$\et$ for clarity. In particular, if $m$ is invertible in $k$, multiplication by $m$ gives a
``Kummer" isomorphism
\[\mu_m\iso \Z/m(1)_\et\]
hence morphisms 
\begin{equation}\label{eq3.1}
\mu_m^{\otimes n}\to \Z/m(n)_\et
\end{equation}
in the category $\DM_\et^\eff(k)$. The following reformulation of \eqref{eq3.3} is a
special case of Theorem \ref{t3.2} b) (see also \cite[Th. 6.1]{VIHES}):

\begin{thm}\label{t3.4} \eqref{eq3.1} is an isomorphism for any $n\ge 0$.
\end{thm}

\subsection{The $l$-adic realisation functor} Let us mention immediately that, for a prime
number
$l$ invertible in $k$, there exists an $l$-adic realisation $\otimes$-functor
\[R_l:\DM_\et(k)\to \hat{D}_\et(k,\Z_l)\]
where the right hand side is the category defined by Ekedahl in \cite{ek}. Intuitively, its
restriction to $\DM^\eff_\et(k)$ is obtained as the ``inverse limit" of the compositions
\[\DM^\eff_\et(k)\to \DM^\eff_\et(k,\Z/l^n)\simeq D(k_\et,\Z/l^n)\]
cf. Theorem \ref{t3.2} b). This functor
commutes with infinite direct sums, which implies formally that it has a right adjoint
$\Omega_l$. When $k$ is a finite field, the motivic reformulation of Friedlander's conjecture
will rely on this pair of adjoints and its relationship with the subcategory $\DM^\eff_\et(k)$.

For details on the definition of $R_l$, I refer to the next section.

\section{Categories of motives over a base}\label{s.base}

In this section, $S$ is a Noetherian separated scheme.

\subsection{Generalising Voevodsky's definition over a field} This pre\-sents no basic
difficulty: in short, one replaces the category $\Cor(k)$ used in the previous section by a
category $\Cor(S)$ based on the theory of relative cycles developed by Suslin and Voevodsky in
\cite{suvo1}. The desciption of $\Cor(S)$ is outlined in \cite[App. 1A]{mvw}, and
in more detail in \cite[\S 2]{voecan}. The construction of a category $\DM_-^\eff(S)$ (even for
$S$ a simplicial scheme!) is then given in \cite{voebase}. 

Ivorra's exposition \cite{ivorra-exp} develops this a bit further in
\S 2, and introduces a corresponding category of geometric motives $\DM_\gm^\eff(S)$ in \S 4:
the only difference with the case of a field is that the Mayer-Vietoris relation is generalised
to ``Mayer-Vietoris for the Nisnevich topology". Finally, Ivorra constructs a fully faithful functor
$\DM_\gm^\eff(S)\Inj
\DM_-^\eff(S)$ in \S 5, just as in the case of a field.

Another exposition of the theory of finite correspondences, in the case of a regular
base, is given by D\'eglise in \cite{deg-exp}. He then gives a general exposition with Cisinski
in
\cite[\S 8]{cis-deg}, that they use in \emph{loc. cit.}, \S 10 to construct categories
$\DM_\gm^\eff(S)$ and
$\DM^\eff(S)$ as above, but also stable versions $\DM_\gm(S)$ and $\DM(S)$. (In \cite{cis-deg},
the index $\gm$ is replaced by $c$.) In this way, the picture of \S\S \ref{s.motgm} and
\ref{s.moteff} is extended to a general base. The only difference is that, in the diagram
\[
\begin{CD}
\DM_\gm^\eff(S) @>a>> \DM^\eff(S)\\
@VcVV @VdVV\\
\DM_\gm(S) @>b>> \DM(S)
\end{CD}
\]
generalising \eqref{eq4.2}, the horizontal functors are still fully faithful but it  is not
known whether the vertical functors are.

\subsection{Where the buck stops}\label{s.bucks} So far we seem to live in the best possible
world, having fully extended Voevodsky's theory over a base. Unfortunately, one runs into
trouble in two situations:

\begin{enumerate}
\item When trying to apply to this construction the yoga
of six operations of Voevodsky-Ayoub.
\item Similarly for realisation functors, most importantly the $l$-adic realisation functors.
\end{enumerate}

In (1), the aim would be to endow $S\mapsto \DM(S)$ with the six functors of
Gro\-then\-dieck:
$f^*,f_*,f_!,f^!$,
$\otimes,
\uHom$ where $f:S'\to S$ is a morphism of Noetherian schemes (cf. \cite[\S 2]{illusie}). The
strategy of Grothendieck, Verdier and Deligne for coherent and
\'etale cohomology is awkward here, mainly because Voevodsky works with big sites rather than
small sites.  For this reason, Voevodsky proposed another axiomatic strategy which fits
better in his context. 

Roughly speaking, his framework is the following. One starts from associating to any
(reasonable) scheme
$S$ a triangulated category $\bH(S)$; one assumes that some of the operations are already
constructed and verify certain axioms. From this input, one can then deduce the other
operations and all their properties.  

After Voevodsky sketched this programme for $f^*,f_*,f_!,f^!$, it was worked out by Ayoub in
\cite[Ch. 1]{ayoub} (see his 1.4.1 for a precise statement of the above).

Voevodsky had in mind primarily the stable homotopy categories of sche\-mes $\SH(S)$ that he
developed jointly with Morel \cite{mv,VICM}. In \cite[Ch. 4]{ayoub}, Ayoub gives examples of
theories $S\mapsto \bH(S)$ which verify his axioms 1.4.1. They include $\SH$, but not $\DM$.

The problem with $\DM$ is that transfers make it
difficult to verify the axiom of ``locality". This axiom (number 4 in
\cite[1.4.1]{ayoub}) roughly requires that if $Z$ is a closed subset of $S$ with open
complement $U$, one can obtain $\bH(S)$ by glueing $\bH(Z)$ and $\bH(U)$.

The issue for (2): constructing $l$-adic realisation functors, is to extend the rigidity theorem
of Suslin-Voevodsky over a base. For this, a strategy to reduce to the case of a perfect base
field treated in \cite{suvo} is obstructed by the same axiom of locality.

\subsection{Where the buck restarts: motives without transfers}\label{s.buck} On the other hand,
the examples of \cite[Ch. 4]{ayoub} do include \emph{motives without transfers}. If $\tau$ is a
suitable topology, the category $\DA_\tau(S)$ is constructed as $\DM_\tau(S)$, except that one
replaces the category $\Cor(S)$ by the simpler category $\Sm(S)$ (more correctly:
$\Z[\Sm(S)]$) throughout.  We shall use the established practice of dropping $\tau$ if it is the
Nisnevich topology. For $S$ the spectrum of a field $k$, this construction appears first
in \cite[\S 5.2]{morel} with the notation
$\widetilde{\DM}(k)$.

That this theory satisfies locality is the contents of
\cite[Cor. 4.5.44]{ayoub}\footnote{The notation of \cite{ayoub} may be confusing: compare
\cite[App. A]{ayoubreal}.}, whose proof is inspired by that of
\cite[Th. 2.21]{mv}. An essentially self-contained construction of $\DA_\tau$ is given in
\cite[\S 3]{ayoubreal}.

As above, one may define this theory with coefficients in a commutative ring $A$; there is a
natural $\otimes$-functor
\[\DA_\tau(S,A)\to \DM_\tau(S,A).\]

In \cite[\S 15.2]{cis-deg}, Cisinski and D\'eglise prove that this functor is an equivalence of
categories when $A=\Q$ and $\tau=\et$ if $S$ is excellent and geometrically unibranch. In \cite[Th. B.1]{ayoubreal}, Ayoub
extends this to the case where any prime number is invertible either in $A$ or
on
$S$, provided
$S$ is universally Japanese, normal and of finite \'etale cohomological dimension. (There is an
extra minor condition, and $2$ must be inverted unless the characteristic is $0$.) In this way,
the formalism of the 4 operations is largely repaired.

If $S$ has characteristic $p>0$, we may choose $A=\Z[1/2p]$. But the Artin-Schreier argument of Theorem \ref{t3.2} c)
shows that 
\[\DM_\et(S)\to \DM_\et(S,\Z[1/p])\] 
is an equivalence of categories, and the same
works for $\DA_\et$. So the comparison theorem holds with coefficients $\Z[1/2]$ in this case.
If $S$ is a $\Q$-scheme, then it holds with integer coefficients.

\subsection{A competitor: Beilinson motives} In \cite{cis-deg}, Cisinski and D\'eglise do prove
the locality axiom outlined in \S \ref{s.bucks} for $\DM$ when $S$ and $Z$ are smooth over a
common base. They can relax this smoothness assumption provided one works with rational
coefficients. To this aim, they introduce yet another theory: \emph{Beilinson motives}.

Roughly speaking, Beilinson motives are modules over the $K$-theory spectrum, an
object of $\SH$ which represents algebraic $K$-theory. They form a $\Q$-linear tensor
triangulated category $\DM_{\cyr{B}}(S)$, defined for any reasonable $S$. The locality axiom for
$\SH$ then implies rather formally the locality axiom for $\DM_{\cyr{B}}$.

They also construct an equivalence of $\otimes$-triangulated categories
\[\DM_{\cyr{B}}(S)\simeq \DM(S,\Q)\] 
for any excellent and geometrically unibranch scheme $S$
\cite[Th. 8 and 9]{cis-deg}, which extends the locality axiom for $\DM(-,\Q)$ to these schemes.

Finally, they prove an equivalence 
\[\DM_{\cyr{B}}(S)\simeq \DA_\et(S,\Q)\] 
for any $S$ \cite[Th. 13]{cis-deg}, which connects with the results explained in \S
\ref{s.buck}.

\subsection{Recapitulation}\label{s.recap}

Since the above discussion may have confused the reader, let me summarise it:

\begin{enumerate}
\item There are three types of triangulated theories of motives on the market, associated to a
reasonable scheme $S$: $\DM_\tau(S,A)$, $\DA_\tau(S,A)$, $\DM_{\cyr{B}}(S)$. Here $\tau$ denotes
a suitable Grothendieck toplogy (usually Nisnevich or \'etale) and $A$ denotes a commutative
ring. The categories $\DM_{\cyr{B}}(S)$ are $\Q$-linear.

(This does not exhaust the triangulated theories of motives on the market: there are also the
versions of Levine and Hanamura, see \S \ref{0.1}. Here I limit myself to those which are
constructed in the spirit of Voevodsky and Morel.)
\item For any $S$, there is an equivalence of $\otimes$-categories $\DM_{\cyr{B}}(S)\simeq
\DA_\et(S,\Q)$.
\item Suppose for simplicity that $S$ is a $k$-scheme for a field $k$ of exponential
characteristic $p$. Then the natural $\otimes$-functor
\[\DA_\et(S,\Z[1/2])\to \DM_\et(S,\Z[1/2])\]
is an equivalence of categories provided $S$ is normal and universally Japanese. If $p=1$, this
holds with integer coefficients.
\item The assignment $S\mapsto \DA_\tau(S,A)$ satisfies the formalism of the six operations of
Voevodsky-Ayoub (for $\tau=\Nis, \et$).
\end{enumerate}

\subsection{Some philosophy}  It may be disturbing that one has to replace the beautiful
categories of motives based on finite correspondences by apparently uglier categories in order
to have a working theory. In fact, this is not so serious. What matters is that this theory
coincides with
$\DM_\et(-,A)$ in important cases, especially when the base is a field. In this way, the
connection with algebraic cycles is not lost. As for other bases, I like to think of some
theory $S\mapsto
\bH(S)$ extending $k\mapsto \DM_\et(k,A)$ as a black box, which is helpful for applying the six
functors formalism but whose specific description  for a general $S$ is not so important. An
example is given just below.

\subsection{Application: motives of singular varieties}\label{s.sing} Let $k$ be a field. In
\cite[\S 4.1]{V}, Voevodsky associates to any $k$-scheme  $X$  of finite type two objects
$\underline{C}_*(X)$ and $\underline{C}_*^c(X)$ of $\DM_-^\eff(k)$. These objects play the
respective r\^oles of the motive and motive with compact supports of $X$. Unfortunately,
the proof of their functoriality rests on \cite[Th. 4.1.2]{V}, whose proof uses Hironaka's
resolution of singularities in a central way, so is valid only in characteristic $0$. Under
this condition, the functorialities show that these objects belong to $\DM_\gm^\eff(k)$, by
reduction to the case of smooth projective schemes.

As I first learned from Jo\"el Riou, the six operations allow us to define such objects in
nonzero characteristic at least in the \'etale categories, by avoiding Hironaka resolution:

\begin{defn}\label{d6.1}  Let $\tau$ be the Nisnevich or the \'etale topology, and let $p_X:X\to
\Spec k$ be a
$k$-scheme of finite type. We define:
\begin{enumerate}
\item Its motive: $M(X)=\uHom((p_X)_*\Z_X,\Z)\in \DA_\tau(k)$.
\item Its motive with compact supports: $M_c(X)=\uHom((p_X)_!\Z_X,\Z)\allowbreak\in\DA_\tau(k)$.
\end{enumerate}
\end{defn}

\begin{prop}\label{p6.2} a) The motive is covariant for any morphism.\\
b) The motive with compact supports is covariant for projective morphisms and contravariant for
\'etale morphisms.\\ 
c) Let $X$ be a $k$-scheme of finite type, $i:Z\to X$ a closed immersion
and
$j:U\to X$ the complementary open immersion. Then the sequence
\[M_c(Z)\by{i_*}M_c(X)\by{j^*}M_c(U)\]
defines an exact triangle in $\DA_\tau(k)$.\\
d) Let 
\[\begin{CD}
Z'@>i'>> X'\\
@VqVV @VpVV\\
Z@>i>> X
\end{CD}\]
be an abstract alteration in the following sense: $p$ is projective, $i$ is a closed immersion,
$Z'=p^{-1}(Z)$ and 
\[r=p_{|X'-Z'}:X'-Y'\to X-Z\] 
is an \'etale (finite) morphism of degree $m$. Then the map of cones
\[\cone(M(Z')\to M(X'))\to \cone(M(Z)\to M(X))\]
is split surjective either in $\DM(k,\Z[1/m])$ or in $\DA_\et(k,\Z[1/m])$ and is an
isomorphism in $\DA_\tau(k,\Z[1/m])$ if
$m=1$. In particular, if
$m=1$ we have an exact triangle
\[M(Z')\to M(X')\oplus M(Z)\to M(X)\by{+1}.\]
e) If $X$ is projective, $M_c(X)\simeq M(X)$.\\
f) If $X$ is smooth, $M(X)$ is the classical motive of $X$.  If we work in $\DM(k)$ or $\DA_\et(k,A)$, with $2$ invertible in $A$, then
\[M_c(X)\simeq M(X)^*(-d)[-2d],\]
where $d=\dim X$.\\ 
g) Let $\sT=\DA_{\gm,\et}^\eff(k)$ or $\DM_\gm^\eff(k,\Z[1/p])$, where $p$ is the exponential
characteristic of $k$. Then $M(X),M_c(X)\in
\sT$ for any $X$.
\end{prop}

\begin{proof} We shall prove the (a priori finer) dual functoriality to a) -- f) for
$M^*(X):=(p_X)_*\Z_X$ and $BM(X):=(p_X)_!\Z_X$.

a) Let $f:X\to Y$. The isomorphism of functors
\[(p_X)_*\simeq (p_Y)_*\circ f_*\]
gives an isomorphism
\[M^*(X)\simeq (p_Y)_*( f_*\Z_X).\]

We get a morphism $f^*:M^*(Y)\to M^*(X)$ from the unit morphism $\Z_Y\to f_*f^*\Z_Y =
f_*\Z_X$.

b) Same argument for the contravariance of $BM$, using the covariance of $f_!$ and the
isomorphism
$f_!\iso f_*$ for $f$ projective \cite[Sch. 1.4.2]{ayoub}. Let now $f:X\to Y$ be an
\'etale morphism. To define a covariant functoriality, we have to exhibit a canonical morphism
$f_!\Z_X\to \Z_Y$ or equivalently
$\Z_X\to f^!\Z_Y$. But since $f$ is \'etale, we have $f^!=f^*$ (ibid.) and we take the identity
morphism.

c) follows from applying $(p_X)_!$ to the exact triangle
\[j_!j^!\Z_X\to \Z_X\to i_*i^*\Z_X\by{+1}\]
from \cite[Lemma 1.4.6]{ayoub}, noting that $i^*\Z_X=\Z_Z$, $j^!\Z_X=j^*\Z_X=\Z_U$ and
$i_*=i_!$.

d) Let $j:X-Z\to X$, $j':X'-Z'\to X'$ be the complementary open immersions to $i,i'$. Applying
to
$p_*\Z_{X'}$ the exact triangle of functors used in the proof of c), we get a commutative
diagram of exact triangles
\[\begin{CD}
j_!j^!p_*\Z_{X'}@>>> p_*\Z_{X'}@>>> i_*i^*p_*\Z_{X'}@>+1>>\\
@AAA @AAA @AAA\\
j_!j^!\Z_X@>>> \Z_X@>>> i_*i^*\Z_X@>+1>>.
\end{CD}\]

By projective base change \cite[Sch. 1.4.2 5]{ayoub}, this may be rewritten as
\[\begin{CD}
j_!r_*\Z_{X'-Z'}@>>> p_*\Z_{X'}@>>> i_*q_*\Z_{Z'}@>+1>>\\
@AAA @AAA @AAA\\
j_!\Z_{X-Z}@>>> \Z_X@>>> i_*\Z_Z@>+1>>
\end{CD}\]
or
\[\begin{CD}
j_!r_*\Z_{X'-Z'}@>>> p_*\Z_{X'}@>>> p_*i'_*\Z_{Z'}@>+1>>\\
@AAA @AAA @AAA\\
j_!\Z_{X-Z}@>>> \Z_X@>>> i_*\Z_Z@>+1>>.
\end{CD}\]

Consider the unit map $a:\Z_{X-Z}\to r_*\Z_{X'-Z'}=r_!\Z_{X'-Z'}$ and
the counit map  $b:r_!\Z_{X'-Z'}\to \Z_{X-Z}$. The composition $ba$ is multiplication by $m$ in
two cases: if we push the situation in $\DM(S)$ (because $b$ is given by the transpose of the
graph of $r$) or if $\tau=\et$ (same proof as for \cite[Lemma 2.3]{ayoubreal}). If
$m=1$, $r$ is the identity in any situation. Applying
$(p_X)_*$ to the above diagram, this concludes the proof of d).

e) follows from the isomorphism $(p_X)_*\iso (p_X)_!$ (when $X$ is projective).

In f), the second statement follows from the isomorphism $(p_X)_!\iso (p_X)_\#(-d)[-2d]$, where
$(p_X)_\#$ is the left adjoint of $p_X^*$, defined because $p_X$ is smooth (\cite[Sch.
1.4.2 3]{ayoub} plus \cite[cor. 2.14]{ayoubreal}); recall that the isomorphism $(p_X)_\# \Z_X=M(X)$ is formal. The first statement
follows from the isomorphism
\[\uHom_k((p_X)_\#A,B)\iso (p_X)_*\uHom_X(A,p_X^*B)\]
from \cite[Prop. 2.3.52]{ayoub}, applied with $A=\Z_X,B=\Z$.

Finally, g) follows from the previous properties by d\'evissage using Gabber's refinement of
de Jong's theorem \cite{illusie-gabber}, by reduction to the case of smooth projective
varieties.
\end{proof}

\begin{remarks}\label{r6.1}
1) The condition ``projective" may be relaxed to ``prop\-er" by arguments involving Chow's
lemma\footnote{plus a support property \cite[2.2.5]{cis-deg}, as pointed out by the referee.}, cf. \cite[Exp. XVII, \S 7]{sga4} and \cite{cis-deg}.

2) There are other useful formulas for $M(X)$ and $M^c(X)$:
\[M(X)=f_!f^!\Z, \qquad M^c(X)=f_*f^!\Z\]
obtained by applying the Verdier duality of \cite[Vol. I, p. 435, Th. 2.3.75]{ayoub}.

3) It would remain to compare the motives $M(X)$ and $M_c(X)$ of Proposition \ref{p6.2} (or
rather their images in $\DM_\tau(k)$) with those defined by Voevodsky. We leave this as a
problem for the interested reader. 

4) This provides definitions for motivic theories associated to a $k$-scheme of finite
type $X$, in the style of
\cite[\S 9]{VF}:
\begin{description}
\item[Motivic cohomology] 
\[H^i_\tau(X,\Z(n)) = \DA_\tau(k)(M(X),\Z(n)[i]).\]
\item[Motivic cohomology with compact supports] 
\[H^i_{\tau,c}(X,\Z(n)) =\DA_\tau(k)(M_c(X),\Z(n)[i]).\]
\item[Motivic homology] 
\[H_i^\tau(X,\Z(n)) = \DA_\tau(k)(\Z(n)[i],M(X)).\]
\item[Borel-Moore motivic homology] 
\[H_i^{\tau,c}(X,\Z(n)) =\DA_\tau(k)(\Z(n)[i],M_c(X)).\]
\end{description}
It would remain to compare these (co)homology groups with those given by cdh cohomology as in
\cite{VF,V} or \'eh cohomology as in \cite{geissereh}: we leave this as another problem for the
interested reader.
\end{remarks}

\subsection{The $l$-adic realisation functor}\label{s.real}  The correct version of the target
is the category defined by Ekedahl in \cite{ek}; technical details are spelled out in Ayoub's
article
\cite[\S 5]{ayoubreal}. He gets $l$-adic realisation functors:

\begin{thm}\label{t5.1} Let $S$ be a $k$-scheme of finite type and $l\ne 2,p$, where $k$ is a
finite field of characteristic $p$. Then there exist  $\otimes$-triangulated functors
\begin{equation}\label{eq5.2}
R_l:\DA_\et(S,\Z_l)\to \hat{D}_\et(S,\Z_l)
\end{equation}
sending the ``motive" of a smooth $S$-scheme $f:X\to S$ to the dual of $Rf_*\Z_l$. We shall also write $R_l$ for its restriction to $\DA_\et(S)$.\\
These
functors commute with the $6$ operations of Grothendieck.
\end{thm}

\begin{proof} See  \cite[Prop. 5.8]{ayoubreal} for \eqref{eq5.2}, \emph{ibid.}, Th. 6.6 and
Prop. 6.7 for the compatibility with the 6 operations and \emph{ibid.}, Prop. 5.9 for the value
of $R_l$ on the motive of $X$.
\end{proof}

\subsection{The motivic version of the $l$-adic cycle class map} Assume that $S=\Spec k$ in
Theorem
\ref{t5.1}. For a smooth
$k$-scheme
$f:X\to
\Spec k$,
$R_l$ induces homomorphisms
\begin{multline}\label{eq5.1}
H^i_\et(X,\Z(n)) = \DM_\et(k)(M(X),\Z(n)[i])\\
\by{R_l} \hat{D}_\et(k,\Z_l)(R_lM(X),R_l\Z(n)[i]) = \hat{D}_\et(k,\Z_l)(\Z_l,Rf_*\Z_l(n)[i])\\
=H^i_\cont(X,\Z_l(n)).
\end{multline}

Here the equality on the first line follows from Theorem \ref{t3.3} a) and the full
faithfulness of the functor $d_\et$ in \eqref{eq4.1}. The equality on the second line uses the
fact that $Rf_*\Z_l$ is a strongly dualisable object: this is a formal consequence of the fact
that $R_l$ is a (strict) $\otimes$-functor and that $M(X)$ is strongly dualisable in
$\DM_\et(k)$ (Lemma \ref{l4.1}).

One can check that \eqref{eq3.2} is induced by \eqref{eq5.1} when $k$ is finite.

We get similar cycle class maps for the (co)homology theories considered in Remark \ref{r6.1}
4).

However, as said in \S \ref{s.naive}, these maps are insufficient to get a good reformulation of Friedlander's conjecture: we shall have to wait until \S \ref{s.rel} to get the latter.

\section{Effectivity and spectra}\label{s.efsp}

This section is very formal and would work in a wider generality; it will be used
in the proof of Corollary \ref{c9.1}. 

\subsection{Compacts, duals and Brown representability} In this subsection I review some now
basic facts about ``large" triangulated categories. First some terminology from \cite{neeman2}:

\begin{defn} Let $\sS$ be a triangulated category.\\
a) An object $X\in \sS$ is \emph{compact} if the functor $\sS(X,-)$ commutes with representable
direct sums.\\
b) A subcategory of $\sS$ is \emph{localising} if it is full, triangulated and stable
under representable direct sums.\\
c) A set of objects $\sX$ of $\sS$ is \emph{dense in $\sS$} if 
\[\sX^\perp:=\{Y\in \sS\mid \sS(X,Y[i])=0\ \forall X\in\sX\ \forall i\in\Z\}\]
is reduced to $0$.\\
d) The category $\sS$ \emph{satisfies TR5} if small direct sums are representable in $\sS$, and
define triangulated functors.\\
e) The category $\sS$ is \emph{compactly generated} if it satisfies TR5 and there exists a dense
set of compact objects.
\end{defn}

Clearly, the full subcategory $\sS^c\subseteq \sS$ consisting of compact objects is triangulated
and \emph{thick}, i.e., closed under direct summands.

The following is Neeman's version of Brown's representability theorem (a special case
of \cite[Prop. 8.4.2]{neeman2}):

\begin{thm}\label{tBrown} Let $\sS$ be a compactly generated triangulated category. Let
$H:\sS^{op}\to
\Ab$ be a cohomological functor. Then $H$ is representable if and only if it converts
small direct sums into small products.
\end{thm}

\begin{cor}\label{cBrown} Let $\sS$ be as in Theorem \ref{tBrown}, and let $f$ be a triangulated
functor to another triangulated category $\sT$. Then $f$ has a right adjoint if and only if it
commutes with small direct sums. This right adjoint is triangulated.
\end{cor}

To deduce Corollary \ref{cBrown} from Theorem \ref{tBrown} is left to the reader as an exercise.

We also note the following useful lemma of Ayoub:

\begin{lemma}[\protect{\cite[Lemma 2.1.28]{ayoub}}]\label{lAyoub} Let $\sS,\sT,f$ be as
in Corollary
\ref{cBrown}, and assume that $f$ has a right adjoint $g$. Then $g$ commutes with small direct
sums if and only if $f$ preserves compact objects. 
\end{lemma}

Finally, if $\sS$ is a tensor triangulated category \cite[Def. 2.1.148]{ayoub}, we have
the notion of strongly dualisable object, see Lemma \ref{l4.1} and its proof.

\begin{lemma}\label{ldual-compact} Let $\sS$ be a tensor triangulated category, with unit object
$\un$. Then strongly dualisable objects are compact if and only if $\un$ is compact.
\end{lemma}

\begin{proof} If $X,Y\in \sS$ with $X$ strongly dualisable, then
\[\sS(X,Y)\simeq \sS(\un, X^*\otimes Y)\]
where $X^*$ is the dual of $X$.
\end{proof}

\subsection{The $\hocolim$ construction} Let $\sS$ be a triangulated category in
which countable direct sums are representable. In \cite{bn}, B\"okstedt and Neeman observed
that one can perform the hocolim construction, up to non unique isomorphism: given a sequence
\begin{equation}\label{eqA.6}
\dots\by{f_{n-1}} X_n \by{f_n} X_{n+1}\by{f_{n+1}}\dots 
\end{equation}
in $\sS$, define $\hocolim(X_n,f_n)$ as a cone of the morphism
\[\bigoplus_{n\in \Z} X_n\by{1-f_\cdot}\bigoplus_{n\in \Z} X_n.\]

If $Y$ is a compact object of $\sS$, one easily constructs an isomorphism
\begin{equation}\label{eq7.5}
\colim \sS(Y,X_n)\iso \sS(Y,\hocolim X_n).
\end{equation}

\subsection{$\Z(1)$-spectra} Let $S$ be a Noetherian scheme and $A$ be a commutative ring. Write
$cd(S,A)$ for the \'etale cohomological dimension of $S$ relatively to sheaves of $A$-modules. 
We need:

\begin{prop}\label{p6.1} Suppose that $cd(S,A)<\infty$.\\ 
a) The category $\DA_\et^\eff(S,A)$ is compactly generated.\\
 b) The ``suspension spectrum" functor $i:\DA_\et^\eff(S,A)\to \DA_\et(S,A)$
has a right adjoint $M\mapsto M^\eff$.\\
c) The category $\DA_\et^\eff(S,A)$ is closed, with internal Hom given by
\[\uHom^\eff(M,N)=\uHom(i(M),i(N))^\eff\]
where $\uHom$ is the internal Hom of $\DA_\et(S,A)$.
\end{prop}

\begin{proof} a) If $cd(S,A)<\infty$, then $cd(X,A)<\infty$ for any $X/S$ of finite type
\cite[Exp. X]{sga4}. For such $X$,
\[\DA_\et^\eff(S,A)(M(X),C)\simeq H^0_\et(X,C)\]
for any $C\in \DA_\et^\eff(S,A)$ viewed as a complex of \'etale sheaves. To see that $M(X)$ is
compact, one reduces to the cohomology of a single sheaf via hypercohomology spectral sequences,
and then one uses that \'etale cohomology commutes with filtering direct limits of sheaves. Clearly, the set of $M(X)$'s\footnote{more correctly, $X$ running through a set of
representatives of isomorphism classes of smooth $S$-schemes} is dense in
$\DA_\et^\eff(k,A)$. Axiom TR5 is also clear in terms of complexes
of sheaves (this is where the unboundedness is important).

b) follows from a) and Theorem
\ref{tBrown}, or directly from the construction of $\DA_\et(S)$ as a category of spectra. To apply it, we use the fact that $i$ commutes with small direct sums. c)
follows formally.
\end{proof}

In $\DA_\et^\eff(S)$, we have the object $\Z(1)$ which may be defined as $\pi^*\Z(1)$ for $\pi:S\to \Spec k$, or alternately as a shifted direct summand of the motive of $\P^1_S$. For an object $M\in \DA_\et^\eff(S)$, we write as usual $M(1):=M\otimes \Z(1)$. 

We may consider ``$\Z(1)$-spectra" in $\DA_\et^\eff(S)$: these are systems 
\[(M_n,\phi_n)_{n\in\Z}, \quad \phi_n:M_n(1)\to M_{n+1}.\]

Of course, one could also consider $\Z(1)[i]$-spectra for some $i\in\Z$: examples which may be
relevant are $i=1$ ($\Z(1)[1]\simeq \G_m[0]$) and $i=2$ $\Z(1)[2]\in \DM_\gm^\eff(k)$ is the
image of the Lefschetz motive). This does not make any difference in practice.

Let $M\in \DA_\et(S)$: we may associate to it the $\Z(1)$-spectrum
\[M_n=(M(n))^\eff.\]

The transition morphisms 
\begin{equation}\label{eq7.0}
\phi_n:M_n(1)\to M_{n+1}
\end{equation}
are obtained by twisting once the counit morphisms
$i M_n\to M(n)$ and using that $i$ is a monoidal functor.

\begin{lemma}\label{l7.2} The morphisms \eqref{eq7.0} induce isomorphisms
\[M_n\iso \uHom^\eff(\Z(1),M_{n+1}).
\]
 In other words, the spectrum $(M_n)$ is an \emph{$\Omega$-spectrum}.
\end{lemma} 

\begin{proof} By an adjunction game, this follows from the fact that $\Z(1)$ is invertible in
$\DA_\et(S)$.
\end{proof}

\subsection{The case of a field} We now assume $S=\Spec k$, where $k$ is a perfect field of
finite \'etale cohomological dimension. 

Let us
apply the hocolim construction to
$X_n=iM_n(-n)\in
\DA_\et(k)$: the transition morphisms
$M_n(1)\to M_{n+1}$ induce morphisms $f_n:X_n\to X_{n+1}$. These morphisms are compatible with
the untwisted counit morphisms $X_n\to M$, hence induce a morphism
\begin{equation}\label{eq7.4}
\hocolim iM_n(-n)\to M.
\end{equation}

\begin{prop} \label{p7.3} Suppose that $S=\Spec k$. Then,  for any $M\in \DA_\et(k)=\DM_\et(k)$,
\eqref{eq7.4} is an isomorphism.
\end{prop}

(I don't know if this is true over a general base; the proof below uses the full faihtfulness of
$i$, which relies on the cancellation theorem.)

\begin{proof}\label{p7.2} Since $\DM_\et(k)$ is generated by $\DM_{\gm,\et}(k)$, it is enough to check that, for any compact object $N\in \DM_{\gm,\et}(k)$, the morphism
\begin{multline*}
\colim\nolimits_n \DM_\et(k)(N, iM_n(-n))\iso \DM_\et(k)(N,\hocolim_n iM_n(-n))\\
\to \DM_\et(k)(N,M)
\end{multline*}
given by \eqref{eq7.5} and \eqref{eq7.4} is an isomorphism.

There exists $n_0\gg 0$ such that $N(n)\in \IM i$ for $n\ge n_0$. For clarity, let us write $N(n)=i N_n$ for $n\ge n_0$. Since $i$ is fully faithful, we get for $n\ge n_0$
\begin{multline*}
\DM_\et(k)(N, iM_n(-n))=\DM_\et(k)(iN_n, iM_n)\\
 \osi  \DM_\et^\eff(k)(N_n, (M(n))^\eff)\simeq \DM_\et(k)(iN_n, M(n))\\
 =\DM_\et(k)(N(n), M(n))\osi \DM_\et(k)(N,M).
\end{multline*}

We thus get an isomorphism
\begin{multline*}
\DM_\et(k)(N,M)\iso \colim\nolimits_{n\ge n_0} \DM_\et(k)(N, iM_n(-n))\\
=\colim\nolimits_n \DM_\et(k)(N, iM_n(-n)).
\end{multline*}

One check that this isomorphism is inverse to the above map, which concludes the proof.
\end{proof}

\subsection{$\Omega$-spectra and $\Sigma$-spectra} The following corollary is the \emph{raison
d'\^etre} of this section:

\begin{cor}\label{c7.1} For $M\in \DM_\et(k)$, the following are equivalent:
\begin{thlist}
\item $M\in \DM_\et^\eff(k)$.
\item For any $n\ge 0$, the natural morphism $M^\eff(n)\to (M(n))^\eff$ is an isomorphism.
($(M(n))^\eff$ is a ``$\Sigma$-spectrum".)
\item For any $n>0$, $i(M(n))^\eff(-1)\in \DM_\et^\eff(k)$. 
\end{thlist}
\end{cor}

\begin{proof} (i) $\Rightarrow$ (ii) $\Rightarrow$ (iii) is obvious; (ii) $\Rightarrow$ (i)
follows from  Proposition \ref{p7.3} applied to $M$. Finally, assume (iii) true: thus we may
write
$(M(n))^\eff\simeq N(1)$ for some $N\in \DM_\et^\eff(k)$. Using Lemma \ref{l7.2}, we find
\[(M(n-1))^\eff\iso \uHom^\eff(\Z(1),(M(n))^\eff)\simeq \uHom^\eff(\Z(1),N(1))\osi N\]
or $(M(n-1))^\eff(1)\iso (M(n))^\eff$; by induction on $n$, we get (ii).
\end{proof}

\section{A proof of Theorem \ref{tg}}\label{s.geisser2}

From now on, $\F$ is a finite field.

The aim of this section is to give the following effective version of Geisser's theorem
\ref{tg} \cite[Th. 3.3]{geisser0}:

\begin{thm}\label{tG} Let $X$ be a smooth projective $\F$-variety of dimension $d$, and let
$n\ge 0$. Assume:
\begin{enumerate}
\item Conjecture \ref{cb} holds for $X\times X$ in codimension $d$.
\item The cohomological Tate conjecture holds for $X$ in codimension $n$. 
\end{enumerate}
Then, for any $i\ne 2n$, $H^i(X,\Q(n))=0$ (``Parshin's conjecture in weight $n$").
\end{thm}

\begin{proof} We shall use Theorem \ref{t3.1} as well as the functor
\begin{equation}\label{eq8.1}
\Phi:\sM_\rat^\eff(\F,\Q)\to \DM_\gm^\eff(\F,\Q)
\end{equation}
of \cite[Prop. 2.1.4]{V}, where the left hand side is the category of effective Chow motives
over $\F$. 
This shows that
$\End(h(X))=CH^d(X\times X)\otimes
\Q$ acts on the motivic cohomology of
$X$, where $h(X)$ is the Chow motive of $X$. By assumption, 
\[CH^d(X\times X)\otimes \Q\iso A^d_\num(X\times X,\Q).\]

By Jannsen \cite{jannsen}, the right hand side is a semi-simple $\Q$-algebra. Thus we may
decompose $h(X)$ as a direct sum
\[h(X)=\bigoplus_{\alpha\in A} S_\alpha\]
where the $S_\alpha$ are simple motives. Accordingly, we may decompose the motivic cohomology
of $X$ as
\[H^i(X,\Q(n))=\bigoplus_{\alpha\in A} \DM_\gm^\eff(\F,\Q)(\Phi(S_\alpha),\Z(n)[i]).\]

Let $\L$ be the Lefschetz motive. If $S_\alpha=\L^n$, then 
\begin{multline*}
\DM_\gm^\eff(\F,\Q)(\Phi(S_\alpha),\Z(n)[i])=\DM_\gm^\eff(\F,\Q)(\Z(n)[2n],\Z(n)[i])\\
=0\text{ if
} i\ne 2n.
\end{multline*}

To conclude, it remains to see that $\DM_\gm^\eff(\F,\Q)(\Phi(S_\alpha),\Z(n)[i])=0$ if
$S_\alpha\not\simeq\L^n$. 

Let us first show that $\sM_\rat^\eff(\F,\Q)(S_\alpha,\L^n)=0$. This is obvious if $n\notin
[0,d]$. If
$n\in [0,d]$, then $A^n_\num(X)\ne 0$\footnote{If $L$ is an ample line bundle on $X$, then $\deg(L^d)> 0$ and \emph{a fortiori} $0\ne L^n\in A^n_\num(X)$ for all $n\in [0,d]$\dots}, which shows that $\L^n$ appears as a direct summand of
$h(X)$ and we conclude by semi-simplicity.

We now use that Frobenius defines an endomorphism of the
identity functor of $\DM_\gm^\eff(\F,\Q)$: this follows from the easy fact that it is natural with respect to the action of finite correspondences. In particular, for $M,N\in \DM_\gm^\eff(\F,\Q)$, with
Frobenius endomorphisms $F_M,F_N$, we have the equality $F_M^* = (F_N)_*$ on
$\DM_\gm^\eff(\F,\Q)(M,N)$. Since the Frobenius endomorphism of $\L^n$ equals $q^n$, we get:

\begin{lemma} Let $F_\alpha$ be the Frobenius endomorphism of $S_\alpha$. Then 
$F_\alpha^*$ acts on $\DM_\gm^\eff(\F,\Q)(\Phi(S_\alpha),\Z(n)[i])$ by multiplication by
$q^n$.\qed
\end{lemma}

To conclude, it suffices to show that $F_\alpha\ne q^n$ in
$D_\alpha=\End_{\sM_\num^\eff(\F,\Q)}(S_\alpha)$. (Then $F_\alpha^*-q^n$ will act both invertibly
and by $0$ on the $D_\alpha$-vector space $\DM_\gm^\eff(\F,\Q)(\Phi(S_\alpha),\Z(n)[i])$, and the
latter will have to be $0$.)

The cycle class map of Conjecture \ref{cf} splits into components:
\[\sM_\rat^\eff(\F,\Q)(S_\alpha,\L^n)\to H^{2n}(\bar S_\alpha,\Q_l(n))^G\]
with an obvious abuse of notation on the right. As we saw above, the left hand side is $0$.
By (2), the right hand side is therefore $0$ as well. But this means that $1$ is not an
eigenvalue of the characteristic polynomial of $F_\alpha$ acting on $H^{2n}(\bar
S_\alpha,\Q_l(n))$, which concludes the proof.
\end{proof}

\section{Motivic reformulation of Friedlander's conjecture}\label{sFried}

\subsection{An integral version?}\label{s6.1} Getting a motivic version of Conjecture
\ref{cff} with rational coefficients from \eqref{eq3.2} or \eqref{eq5.1} is easy: one considers
the composition
\begin{equation}\label{eq9.1}
\bar \cl_X^{n,i}:H^i_\et(X,\Z(n))\otimes \Q_l\to H^i_\cont(X,\Q_l(n))\to H^i_\cont(\bar X,\Q_l(n))^G
\end{equation}
and conjectures that it is an isomorphism. What about an integral version?

The first idea, to consider the composition
\[H^i_\et(X,\Z(n))\otimes \Z_l\to H^i_\cont(X,\Z_l(n))\to H^i_\cont(\bar X,\Z_l(n))^G\]
is tempting but looks unreasonable: a special case would imply a na\"\i ve integral version of
the Tate conjecture, and it is well-konwn that the corresponding integral refinement of the
Hodge conjecture has counterexamples, the first ones being due to Atiyah-Hirzebruch using
Godeaux varieties.

Over $\F$, counterexamples to the surjectivity of the map
\[CH^n(X)\otimes \Z_l\to H^{2n}_\cont(\bar X,\Z_l(n))^G\]
also exist. In \cite[Aside 1.4]{milneaim}, Milne hints that the Atiyah-Hirzebruch examples also
work here: this was fully justified in \cite[Th. 2.1]{ct-sz}. These counterexamples persist
after any finite extension of $\F$. More recently, A. Pirutka has given an example
of a smooth projective $\F$-variety
$X$ for which the map
$CH^2(X)\to CH^2(\bar X)^G$ is not surjective \cite{pirutka}: this provides a different kind of
counterexample.

Of course, this rules out the even more na\"\i ve idea that \eqref{eq3.2} and \eqref{eq5.1}
could be isomorphisms. In fact they cannot even be isomorphisms rationally, in view of the
exact sequence \eqref{eq3.4}. If $X$ is smooth projective and $i=2n+1$, one gets a map
\[0=H^{2n+1}(X,\Z(n))\otimes \Q_l\to H^{2n+1}_\cont(X,\Q_l(n))\osi H^{2n}_\cont(\bar
X,\Q_l(n))_G.\]

Here we used \eqref{eq4.3} and Theorem \ref{t3.3} b) for the first vanishing, as well as
\eqref{eq3.4} and Theorem \ref{rh}, which gives the vanishing of
$ H^{2n+1}(\bar X,\Q_l(n))^G$. The equality 
\[\dim_{\Q_l} H^{2n}_\cont(\bar X,\Q_l(n))_G= \dim_{\Q_l} H^{2n}_\cont(\bar X,\Q_l(n))^G\]
shows that $H^{2n}_\cont(\bar X,\Q_l(n))_G$ is in general nonzero.

\subsection{The complex $\Z^c$}\label{s6.2} The simplest example of the above is for $n=0$ and
$X=\Spec \F=\Spec \F_q$. Then
\begin{equation}\label{eq6.1}
H^j_\cont(\F_{q^r},\Z_l(0)) =
\begin{cases} \Z_l&\text{if $j=0,1$}\\
0&\text{else}
\end{cases}
\end{equation}
while 
\[H^j_\et(\F_{q^r},\Q(0)) =
\begin{cases} \Q&\text{if $j=0$}\\
0&\text{else.}
\end{cases}\]

In \eqref{eq6.1}, transition morphisms for passing from $\F_q$ to $\F_{q^r}$ are the
identity on $H^0$ but are multiplication by $r$ on
$H^1$. This gives a computation of the cohomology sheaves of
$R\lim(\Z_l)$, where
$R\lim:\hat{D}_\et(\F,\Z_l)\to D(\F_\et,\Z_l)$ is the total derived functor of $\lim$:
\[\lim\nolimits^j\Z_l =
\begin{cases} 
\Z_l&\text{if $j=0$}\\
\Q_l&\text{if $j=1$}\\
0&\text{else.}
\end{cases}
\]

In \cite[\S 4]{tatesheaf}, I computed the object $R\lim \Z_l$ as follows: there
exists an
object $\Z^c\in D^b(\F_\et,\Z)$, represented by an explicit complex of \'etale sheaves of length
$1$, such that
\[R\lim \Z_l \simeq \Z^c\otimes \Z_l\text{ in } D(\F_\et,\Z_l)\]
for any $l\ne p$ (and even for $l=p$) \cite[Def. 4.1 and Th. 4.6 b)]{tatesheaf}.

Let us repeat the properties of $\Z^c$: its cohomology sheaves are
$\Z$ in degree $0$ and $\Q$ in degree $1$; in particular
\begin{equation}\label{eq7.2}
\Z/n[0]\iso\Z^c\oo^L\Z/n.
\end{equation}

The extension class
$\partial\in \Ext^2_{\F_\et}(\Q,\Z)$ corresponding to $\Z^c$ is nonzero; it can be
factored as follows in $D^b(\F_\et,\Z)$:
\begin{equation}\label{eqpartial}
\begin{CD}
\Q[-1] @>\partial>> \Z[1]\\
@V{\pi}VV @A{\beta}AA\\
\Q/\Z[-1] @>\cdot e>> \Q/\Z.
\end{CD}
\end{equation}

Here $\pi$ is induced by the projection $\Q\to \Q/\Z$, $\beta$ is the integral Bockstein and
$\cdot e$ is cup-product with the canonical generator from \eqref{eq6.1c}.

In particular, $\partial\otimes \Q=0$, yielding an isomorphism
\begin{equation}\label{eq6.1b}
\Z^c\otimes \Q\simeq \Q[0]\oplus \Q[-1].
\end{equation}

Finally, we have
\begin{equation}\label{eq6.1a}
H^j_\et(\F,\Z^c) =
\begin{cases} \Z&\text{if $j=0,1$}\\
0&\text{else.}
\end{cases}
\end{equation}

In particular, \eqref{eq6.1c} refines to a generator
\begin{equation}\label{eq6.1d}
e\in H^1_\et(\F,\Z^c).
\end{equation}

See Appendix \ref{s.geisser} for more details on $\Z^c$.

I used $\Z^c$ in \cite{tatesheaf} and \cite{glr} to define a
\emph{modified cycle class map}:
\begin{equation}\label{eq7.3}
H^j_\et(X,\Z(n)\oo^L\Z^c)\otimes \Z_l\overset{\cl_X^{n,j}}{\longrightarrow}
H^{j}_\cont(X,\Z_l(n))
\end{equation}
which serves for an integral version of Friedlander's conjecture. In the sequel, we shall
recover this map more conceptually.

\subsection{A lemma of Ayoub}

Let $\sM,\sN$ be two unital symmetric monoidal categories, and let $f:\sM\to \sN$ be a (strong,
symmetric, unital) monoidal functor. Assume that $f$ has a right adjoint $g$. Then the unit map
\[M\to gf(M)\]
enriches into a bi-natural transformation
\begin{equation}\label{eq6.0}
M\otimes gf(M')\to gf(M\otimes M'), \quad M,M'\in \sM
\end{equation}
which can be described as follows: by adjunction, such a morphism corresponds to a morphism
\[f(M)\otimes fgf(M')\simeq f(M\otimes gf(M'))\to f(M\otimes M')\simeq f(M)\otimes f(M').\]

This morphism is simply induced by the counit map $fgf(M')\to f(M')$.

In particular, we may take $M'=\un$ and get a natural
transformation
\begin{equation}\label{eq6.2}
M\otimes gf(\un)\to gf(M).
\end{equation}

We are interested in a sufficient condition for \eqref{eq6.2} to be an isomorphism. 
For this, note that we have another bi-natural transformation
\begin{equation}\label{eq6.3}
M\otimes g(N)\to g(f(M)\otimes N), \quad M\in \sM,N\in \sN
\end{equation}
which corresponds by adjunction to the map
\[f(M)\otimes fg(N)\simeq f(M\otimes g(N))\to f(M)\otimes N\]
given by the counit at $N$. To recover \eqref{eq6.0} from \eqref{eq6.3}, just take $N=f(M')$.
The following ``projection formula" is due to Joseph Ayoub \cite[Lemme
2.8]{ayoubbetti}:

\begin{lemma} \label{l9.2}
\eqref{eq6.3} is an isomorphism if $M$ is strongly dualisable.
\end{lemma}

\begin{proof} It is so simple that we don't resist in reproducing it:

Given an objet $X$ of $\sM$ or $\sN$, write $t_X$ for the functor $X\otimes -$. Let $C$ be a
strong dual of $M$, i.e., there exist morphisms $C \otimes M \to \un$ and $\un \to M \otimes
C$ which make $t_C$ a left adjoint of $t_M$. As $f$ is monoidal, symmetric and unital, one
deduces that $f(C)$ is a strong dual of $f(M)$, hence that $t_{f(C)}$ is left adjoint to
$t_{f(M)}$. On the other hand, there is an obvious isomorphism  $t_{f(C)}
\circ f \simeq f \circ t_C$. Using the adjunctions $(f; g)$, $(t_C; t_M)$ and
$(t_{f(C)};t_{f(B)})$, one deduces an isomorphism of functors $t_M \circ g \simeq g \circ
t_{f(M)}$. It remains to see that this isomorphism coincides with \eqref{eq6.3}. This
verification is omitted.
\end{proof}

\begin{cor}\label{c6.1} If $M$ is strongly dualisable, \eqref{eq6.2} is an isomorphism.\qed
\end{cor}

\begin{remark}\label{r9.1} Of course Lemma \ref{l9.2} and Corollary \ref{c6.1} can be extended:
suppose that
$\sM,\sN$ are triangulated categories and $f$ (hence $g$) is a triangulated functor. Suppose
also that $g$ commutes with small direct sums (compare Lemma \ref{lAyoub}). Then, given $N\in
\sN$, the full subcategory of $\sM$ consisting of those $M$ for which \eqref{eq6.3} is an
isomorphism is localising. In particular, if strongly dualisable objects are dense in $\sM$,
then \eqref{eq6.3} holds universally as soon as $g$ commutes with small direct sums.
\end{remark}

\subsection{A stable theorem} We shall need:

\begin{lemma}\label{l5.1} The functor $R_l$ of Theorem \ref{t5.1} has a right adjoint
$\Omega_l$.
\end{lemma}

\begin{proof} This is true for both versions: with coefficients $\Z$ or $\Z_l$. Once again this
follows from Corollary \ref{cBrown} (note that $R_l$ commutes with
small direct sums).
\end{proof}

\begin{remark} The unit object $\Z_l$ of $\hat{D}_\et(S,\Z_l)$ is not compact. This implies that
$\Omega_l$ does not commute with small direct sums (see lemma \ref{lAyoub}).
\end{remark}


We now apply Corollary \ref{c6.1} and Remark \ref{r9.1} to the adjunction
$(R_l,\Omega_l)$ of Lemma \ref{l5.1}; we get:

\begin{thm}\label{t6.1} For any strongly dualisable object $M$ of $\DA_\et(S)$,
the map
\[M\otimes \Omega_l(\Z_l)\to \Omega_lR_l(M)\]
is an isomorphism.\\ 
If $S=\Spec \F$ and $l\ne 2$, this applies to all objects of $\DM_{\gm,\et}(\F)$.
\end{thm}

\begin{cor}\label{c6.2} Let $\Gamma=\Omega_l(\Z_l)$. Let $M,N\in \DA_\et(S)$ and suppose
that $N$ is strongly dualisable. Then there is a natural isomorphism
\[\DA_\et(S)(M,N\otimes \Gamma)\iso \hat{D}_\et(S,\Z_l)(R_l(M),R_l(N)).\]
In particular, assume that $S=\Spec \F$ and $l\ne 2$; taking $M=M(X)$ for a smooth $\F$-scheme
$X$ and
$N=\Z(n)[i]$, we get an isomorphism
\[\DM_\et(\F)(M(X),\Gamma(n)[i])\iso H^i_\cont(X,\Z_l(n)).\]
\end{cor}

\begin{lemma}\label{l9.1} Let $f:S\to \Spec \F$ be a smooth $\F$-scheme of finite type. Then
there is a canonical isomorphism $\Omega_l^S(\Z_l)\simeq f^*\Omega_l^{\Spec \F}\Z_l$.
\end{lemma}

\begin{proof} Since $f$ is smooth, $f^*$ has a left adjoint $f_\sharp$ which commutes with $R_l$ \cite{ayoubreal}. This implies formally that $f^*$ commutes with $\Omega_l$. Since $f^*$ also commutes with $R_l$ (ibid.), the conclusion follows.
\end{proof}

\subsection{Towards the integral Friedlander conjecture}\label{s6.4} Theorem \ref{t6.1} and its
corollary look fantastic, but they are in some sense a \emph{trompe-l'\oe il}: all the hard
$l$-adic information is concentrated in the object
$\Gamma = \Omega_l(\Z_l)$. To explain this, let us evaluate this object against the motive of
some smooth variety detwisted: for simplicity we assume $S=\Spec \F$. Thus, if $M=M(X)(-n)[-i]$
with
$X$ a smooth $\F$-scheme and $n,i\in\Z$, we have
\begin{multline*}
\DM_\et(\F)(M(X)(-n)[-i],\Omega_l(\Z_l))=\hat{D}_\et(\F,\Z_l)(R_l(M(X)(-n)[-i]),\Z_l)\\
=\hat{D}_\et(\F,\Z_l)(\Z_l,Rf_*\Z_l(n)[i])=H^i_\cont(X,\Z_l(n))
\end{multline*}
where $f:X\to \Spec \F$ is the structural morphism.

Nevertheless, Corollary \ref{c6.2} gives the flavour of the integral Friedlander conjecture,
to come. In order to get something interesting, we need to go to the
\emph{effective} subcategory
$\DM_\et^\eff(\F)$. 

\subsection{The modified motivic cycle map} \label{s.rel}

From Lemma \ref{l5.1} and Proposition \ref{p6.1}, we deduce that the composite functor
\[\DA_\et^\eff(S)\by{i} \DA_\et(S)\by{R_l}\hat{D}_\et(S,\Z_l)\]
has the right adjoint
\[\Omega_l^\eff:\hat{D}_\et(S,\Z_l)\by{\Omega_l}\DA_\et(S)\by{-^\eff}\DA_\et^\eff(S).\]

Using \eqref{eq6.2},
we get a morphism for any $M\in \DA_\et^\eff(S)$:
\begin{equation}\label{eq6.4}
M\otimes \Gamma^\eff\to \Omega_l^\eff R_l(M)
\end{equation}
where $\Gamma=\Omega_l(\Z_l)$, as in Corollary \ref{c6.2}. The most important case is
$M=\Z(n)$ for $n\ge 0$, which we now record:
\begin{equation}\label{eq6.5}
\Gamma^\eff(n)\to \Omega_l^\eff \Z_l(n).
\end{equation}

The right hand side of \eqref{eq6.5} is denoted by $\Z_l(n)^c$ in \cite{tatesheaf} and
\cite{glr}. It computes continuous \'etale cohomology with coefficients in $\Z_l(n)$.

\begin{defn}\label{d7.1} We call \eqref{eq6.5} the \emph{modified motivic cycle map}.
\end{defn}

\subsection{The object $\Gamma^\eff$} From now on, we assume that $S=\Spec \F$ and $l\ne
2$.\footnote{In order to work with $\DM_\et(\F)$ rather than $\DA_\et(\F)$.} Of course,
Corollary
\ref{c6.1} tells us that
\eqref{eq6.4} is an isomorphism when $M$ is strongly dualisable in $\DM_\et^\eff(\F)$, but there
are very few such objects: we only have the objects of $d_{\le 0} \DM_{\gm,\et}^\eff(\F)$ (see
below). The motive $\Z(n)$, for $n>0$, is certainly not of this kind. Nevertheless we are in a
better situation than in \S
\ref{s6.4}, in that it is possible to compute
$\Gamma^\eff$:

\begin{prop}\label{p7.1} We have a canonical isomorphism
\begin{equation}\label{eq7.1}
\Z^c\otimes \Z_l\iso \Gamma^\eff
\end{equation}
where $\Z^c$ is the object of $D^b(\F_\et,\Z)$ recalled in \S \ref{s6.2}. In particular,
$\Gamma^\eff\otimes \Q\simeq \Q_l[0]\oplus \Q_l[-1]$.
\end{prop}

\begin{proof} This corresponds to \cite[Th. 6.3 and Cor. 6.4]{tatesheaf}. Let me repeat the
proof, which is a simple application of Theorem \ref{rh}. 
We first construct the map \eqref{eq7.1}. For $n\ge 0$, let $d_{\le n} \DM_\et^\eff(\F)$ be the
localising subcategory of $\DM_\et^\eff(\F)$ generated by motives of varieties of dimension $\le
n$ \cite[\S 3.4]{V}. The inclusion functor $i_n:d_{\le n} \DM_\et^\eff(\F)\to \DM_\et^\eff(\F)$
has a right adjoint $\rho_n$. Suppose $n=0$: by \S \ref{s6.2} we have
\[\rho_0 \Gamma^\eff \simeq \Z^c\otimes \Z_l\]
and we get the map by adjunction.

It is clear that 
\[\Gamma^\eff\otimes \Z/l\simeq \Omega_l^\eff(\Z/l)=\Z/l.\]

Combining this with \eqref{eq7.2}, we find that the cone of \eqref{eq7.1} is uniquely
divisible. Hence it suffices to show that \eqref{eq7.1}$\otimes \Q$ is an isomorphism.

Let $X$ be a
smooth projective
$\F$-variety, and let $X\to\pi_0(X)$ be the Stein factorisation of the structural morphism $X\to
\Spec \F$: $\pi_0(X)$ is the ``scheme of constants" of $X$. For $i\in\Z$, we have a commutative
diagram
\[\begin{CD}
H^i_\et(\pi_0(X),\Z^c\otimes \Q_l)@>a>> H^i_\et(X,\Z^c\otimes \Q_l)\\
@VcVV @VdVV\\
H^i_\cont(\pi_0(X),\Q_l)@>b>> H^i_\cont(X,\Q_l)
\end{CD}\]
where the vertical maps are induced by \eqref{eq7.1}$\otimes \Q$. 

In this diagram, $c$ is an isomorphism by the result of \S \ref{s6.2}. The
 map $a$ is an isomorphism for trivial reasons.  The Riemann hypothesis over finite
fields and \eqref{eq3.4} imply that $b$ is an isomorphism. Hence so is $d$.

Thus \eqref{eq7.1}$\otimes \Q$ induces an isomorphism when evaluated against the motive
$M(X)[-i]$ for any smooth projective $X$ and any $i\in\Z$. Since these objects generate
$\DM_\et^\eff(\F,\Q)$,
\eqref{eq7.1}$\otimes \Q$ is an isomorphism by Yo\-ne\-da's lemma.
\end{proof}

The terminology of Definition \ref{d7.1} is now justified by

\begin{cor}\label{c9.3} Evaluating \eqref{eq6.5} against the motive of a smooth variety $X$
yields a map
\[H^j_\et(X,\Z(n)\oo^L\Z^c)\otimes \Z_l\overset{\cl_X^{n,j}}{\longrightarrow}
H^{j}_\cont(X,\Z_l(n))\] 
of the form \eqref{eq7.3}.\qed
\end{cor}

\begin{remark}\label{r9.7} Let us record the other computations of \cite[Cor. 6.10 and
6.12]{tatesheaf}, which also rely on the Riemann hypothesis:
\begin{enumerate}
\item For $n\in \Z$, the cohomology sheaves of $\Omega_l^\eff\Z_l(n)\otimes \Q$ are concentrated
in
$[n,2n+1]$; in particular, $\Omega_l^\eff\Z_l(n)\otimes \Q=0$ for $n<0$.
\item For $0\le m<n$, $\rho_m(\Omega_l^\eff\Z_l(n)\otimes \Q)=0$ (see proof of Proposition
\ref{p7.1} for the definition of $\rho_m$). 
\end{enumerate}
\end{remark}

\subsection{The motivic Friedlander conjecture}\label{sfried} We start with:

\begin{lemma}\label{l7.1} For any $M\in \DM_\et(\F)$ and $\nu\ge 1$, \eqref{eq6.4}$\otimes
\Z/l^\nu$ is an isomorphism (hence so is \eqref{eq6.5}$\otimes
\Z/l^\nu$).
\end{lemma}

\begin{proof}
Theorem \ref{t3.2} b) implies that $R_l^\eff=R_l\circ i$ and $\Omega_l^\eff$ become
quasi-inverse equivalence of categories when we pass to $\Z/l^\nu$ coefficients.
\end{proof}


\begin{prop}\label{rk=f} Let $X$ be a smooth $\F$-variety and $n\in \Z$. Then the following are
equivalent:
\begin{thlist}
\item The map $\cl_X^{n,j}$ of Corollary \ref{c9.3} is an isomorphism for any $j\in\Z$.
\item Same as {\rm (i)}, replacing $\cl_X^{n,j}$ by $\cl_X^{n,j}\otimes \Q$.
\item The map $\bar \cl_X^{n,j}$ of \eqref{eq9.1} is bijective for any $j\in \Z$
(Friedlander's conjecture), and the composition as in Remark \ref{sn} 
\[\rho_X^{n,j}:H^j_\cont(\bar X, \Q_l(n))^G\Inj H^j_\cont(\bar X,
\Q_l(n))\Surj H^j_\cont(\bar X, \Q_l(n))_G\]
is bijective for any $j\in \Z$.  
\end{thlist}  
\end{prop}

\begin{proof} This is an elaboration of the proof of \cite[Prop. 3.9]{glr}:

The equivalence (i)
$\iff$ (ii) follows from Lemma
\ref{l7.1}. We now prove (ii) $\iff$ (iii).

Using \eqref{eq3.4}, \eqref{eq6.1b} and the commutative diagram \eqref{eq9.2} applied to
$C=Rf_*\Q_l(n))$ for $f:X\to \Spec \F$ the structural morphism, one checks that
$\cl_X^{n,j}\otimes \Q$ fits in a diagram of short exact sequences
\[\begin{CD}
0&\to& H^{j-1}(X,\Q(n))_{\Q_l}&\to& H^j(X,\Q^c(n))_{\Q_l}&\to& H^j(X,\Q(n))_{\Q_l}&\to& 0\\
&&@V{f^{j-1}}VV @V{\cl_X^{n,j}\otimes \Q}VV @V{\bar \cl_X^{n,j}}VV\\
0&\to& H^{j-1}_\cont(\bar X,\Q_l(n))_G&\to& H^j_\cont(X,\Q_l(n))&\to& H^j_\cont(\bar
X,\Q_l(n))^G&\to& 0
\end{CD}\]
where $-_{\Q_l}:=-\otimes_\Q \Q_l$, $\Q^c(n):=\Z(n)\oo^L\Z^c\otimes \Q$.
 Moreover, $f^j=\rho_X^{n,j}\circ\bar\cl_X^{n,j}$.

If $\cl_X^{n,j}\otimes \Q$ is bijective, we get the surjectivity of $\bar \cl_X^{n,j}$
and the injectivity of $f^{j-1}$.  So if $\cl_X^{n,j+1}\otimes \Q$ is also bijective, then $\bar
\cl_X^{n,j}$ and
$\rho_X^{n,j}$ are bijective. 

Conversely, if $\bar \cl_X^{n,j}$, $\rho_X^{n,j-1}$ and $\bar \cl_X^{n,j-1}$ are bijective,
then  $\cl_X^{n,j}\otimes \Q$ is
bijective.
\end{proof}

\begin{thm}\label{t1} For $X$ smooth projective and $n\ge 0$, consider the
following  conditions:
\begin{description}
\item[$A(X,n)$] Conjectures \ref{ct} and \ref{cb} hold for $(X,n)$.
\item[$B(X,n)$] $\cl_X^{n,j}$ is an isomorphism for all $j\in\Z$.
\end{description}
Let $d=\dim X$. Then 
\begin{align*}
B(X,n)+B(X,d-n)&\Rightarrow A(X,n)+A(X,d-n)\\
\intertext{and}
A(X\times X,d)+A(X,n)&\Rightarrow B(X,n).
\end{align*}
In particular, the validity of
$B(X^m,n)$ for all $(m,n)$ is independent of $l$.
\end{thm}

\begin{proof} This is a $\DM$ version of the equivalence between conjectures 3.1 and 3.2 of
\cite{glr}. Let me recall the proof:

1) By Proposition \ref{rk=f} and Theorem \ref{rh}, $B(X,n)$
is equivalent to the following three conditions put together:
\begin{thlist}
\item $H^i(X,\Q(n))=0$ for $i\ne 2n$.
\item Conjecture \ref{cf} holds for $(X,n)$.
\item Condition $S^n$ of Remark \ref{sn} holds for $X$.
\end{thlist}

2) By Theorem \ref{tt} and Remark \ref{sn}, (ii) + (iii) for $(X,n)$ and $(X,d-n)$ $\iff$
$A(X,n)+A(X,d-n)$, where $d=\dim X$. In particular we get the first implication in Theorem
\ref{t1}.

3) It remains to see that $A(X\times X,d)+ A(X,n) \Rightarrow$ (i): this follows from Theorem
\ref{tG}.
\end{proof}

\begin{cor}\label{c9.1} a) The following are equivalent:
\begin{thlist}
\item Conjectures \ref{ct} and \ref{cb}.
\item $\cl_X^{n,j}$ is an isomorphism for all smooth projective $X$ and any $n\ge 0$, $j\in\Z$.
\item $\cl_X^{n,j}$ is an isomorphism for all smooth $X$ and any $n\ge 0$, $j\in\Z$.
\item \eqref{eq6.5} is an isomorphism for any $n\ge 0$.
\item For any $n>0$, $(\Omega_l^\eff\Z_l(n))(-1)\in \DM_\et^\eff(\F)$.
\item $\Gamma\in \DM_\et^\eff(\F)$; equivalently, $\Gamma^\eff\iso \Gamma$.
\item Same as {\rm (vi)}, replacing $\DM_\et(\F)$ by $\DM_\et(\F,\Z_l)$.
\end{thlist}
b) If this holds, then \eqref{eq6.4} is an isomorphism for any $M\in \DM_{\gm,\et}^\eff(\F)$.
In particular, for such $M$ and for any $N\in\DM_\et^\eff(\F)$, there is an isomorphism
\[\DM_\et^\eff(\F)(N,M\otimes \Z^c)\otimes \Z_l\iso \hat{D}_\et(\F,\Z_l)(R_l(N),R_l(M)). \]
\end{cor}

\begin{proof} a) (i) $\iff$ (ii): this follows from Theorem \ref{t1}.

(ii) $\iff$ (iv): this is
formal, since $\DM_{\gm,\et}^\eff(\F)$ is generated by the $M(X)$ for $X$ smooth
projective.\footnote{One may avoid the use of Gabber's theorem here by using the equivalence
(i) $\iff$ (ii) of Proposition 
\ref{rk=f} and working in $\DM_\gm(\F,\Q)$: then only de Jong's theorem is needed.}

(iv) $\Rightarrow$ (iii) $\Rightarrow$ (ii) is trivial.

(iv) $\iff$ (v) $\iff$ (vi): this follows from Corollary \ref{c7.1}. 

(vi) $\iff$ (vii) is formal. (Note that the right adjoint to $\DM_\et(\F)\to \DM_\et(\F,\Z_l)$
is just given by restriction of scalars.)

b) Applying the functor $-^\eff$ to the isomorphism of Theorem \ref{t6.1}, we get an
isomorphism
\[(M\otimes \Gamma)^\eff \iso \Omega_l^\eff R_l(M).\]

If (vi) holds, the obvious isomorphism $(M\otimes \Gamma^\eff)^\eff\iso M\otimes
\Gamma^\eff$ yields an isomorphism $M\otimes
\Gamma^\eff\iso\Omega_l^\eff R_l(M)$, that one checks is
\eqref{eq6.4}.
\end{proof}

The next corollary follows from Lemma \ref{l9.1}:

\begin{cor}\label{c9.2} Suppose that the equivalent conditions of Corollary \ref{c9.1} hold,
and let $S$ be a smooth $\F$-scheme of finite type. Then there is a canonical  isomorphism
\[M\otimes \Z_l\otimes \Z^c\iso \Omega_lR_l(M)\]
for any strongly dualisable $M\in \DA_\et(S)$.\qed
\end{cor}

\section{Duality and finite generation}\label{s.fg}

Motivated by the last isomorphism of Corollary \ref{c9.1}, we define here a new category
$\DA_W(\F)$ out of $\DA_\et(\F)$ by using the product $\mu:\Z^c\oo^L \Z^c\to \Z^c$ from Appendix
\ref{s.geisser}. We shall then prove that the equivalent statements of corollary \ref{c9.1} are
also equivalent to the finite generation (as $\Z[1/p]$-modules) of the Hom groups for the
constructible objects of $\DA_W(\F)$.

\subsection{$l$-adic duality}\label{s.ladic} Let $\hat{D}^b_c(\F,\Z_l)$ be the full
subcategory of
$\hat{D}_\et(\F,\Z_l)$ consisting of perfect complexes: it is equivalent to the derived
category of bounded complexes of finitely generated $\Z_l[[G]]$-modules. A form of
Verdier duality is the following (cf.
\cite[Ch. I, Ann. 2]{cg}):

\begin{prop}\label{p.verdier} The ``global sections" functor $R\Gamma:\hat{D}^b_c(\F,\Z_l)\to
D^b(\Z_l)$ has a right adjoint $\Pi$, and $\Pi(X)\simeq X[1]$ where $X\in D^b(\Z_l)$ is viewed
as a complex of constant sheaves in $\hat{D}^b_c(\F,\Z_l)$.\qed
\end{prop}

\begin{cor}\label{c.verdier} Let $C\in \hat{D}_c^b:=\hat{D}_c^b(\F,\Z_l)$. Then\\
a) The composition 
\[\Hom_{\hat{D}_c^b}(\Z_l,C)\times \Hom_{\hat{D}_c^b}(C,\Z_l[1])\to \Hom_{\hat{D}_c^b}(\Z_l,\Z_l[1])=\Z_l\] 
is a perfect duality modulo torsion.\\
b) The compositions 
\[\Hom_{\hat{D}_c^b}(\Z_l,C)\times \Hom_{\hat{D}_c^b}(C,\Z/l^\nu[1])\to
\Hom_{\hat{D}_c^b}(\Z_l,\Z/l^\nu[1])=\Z/l^\nu\] 
induce a perfect duality of finite groups
\[\Hom_{\hat{D}_c^b}(\Z_l,C)_\tors\times \Hom_{\hat{D}_c^b}(C,\Z_l[2])_\tors\to \Q_l/\Z_l.\]
\end{cor}

\begin{proof} Proposition \ref{p.verdier} yields an isomorphism
\[R\Gamma\uHom_\F(C,\Z_l[1])\iso \uHom(R\Gamma(C),\Z_l)\]
where $\uHom_\F$ denotes the internal Hom of $\hat{D}^b_c(\F,\Z_l)$. Taking
cohomology, we get
\[\Hom_{\hat{D}_c^b}(C,\Z_l[q+1])\iso H^q(\uHom(R\Gamma(C),\Z_l))\]
and we conclude with the well-known exact sequences
\begin{multline*}
0\to \Ext^1_{\Z_l}(\Hom_{\hat{D}_c^b}(\Z_l,C[1-q]),\Z_l)\to H^q(\uHom(R\Gamma(C),\Z_l))\\
\to\Hom_{\Z_l}(\Hom_{\hat{D}_c^b}(\Z_l,C[-q]),\Z_l)\to 0
\end{multline*}
for $q=0,-1$. (We leave it to the reader to verify that the corresponding isomorphisms are
really induced by the pairings in a) and b).)
\end{proof}

\subsection{An abstract construction} Let $\sM$ be a unital, symmetric monoidal category and
let $\Gamma$ be a commutative monoid in $\sM$ \cite[Ch. VII, \S 3]{mcl}. We shall construct a
new (unital, symmetric) monoidal category $\sM(\Gamma)$ and a pair of adjoint functors
\[\sM\begin{smallmatrix}\gamma^*\\ \leftrightarrows\\ \gamma_*\end{smallmatrix} \sM(\Gamma)\]
with $\gamma^*$ symmetric monoidal.

For this, we observe that the endofunctor $M\mapsto M\otimes \Gamma$ has the structure of a
monad \cite[Ch. VI, \S 1]{mcl}; we define $\sM(\Gamma)$ as the Kleisli category associated to
this monad (ibid., \S 5, Th. 1). Recall its definition:

\begin{itemize}
\item The objects of $\sM(\Gamma)$ are the objects of $\sM$.
\item Given two objects $M,N$,
\[\sM(\Gamma)(M,N)=\sM(M,N\otimes \Gamma).\]
\item Composition is defined by using the multiplication of $\Gamma$.
\end{itemize}

This construction comes with an adjunction $(\gamma^*,\gamma_*)$ as required. Let us describe
these two functors: 

\begin{itemize}
\item $\gamma^*$ is the identity on objects and acts on morphisms via the unit of $\Gamma$.
\item $\gamma_*$ sends an object $M$ to $M\otimes \Gamma$ and acts on morphisms via the
multiplication of $\Gamma$.
\end{itemize}

The only thing which is not in Mac Lane's book is the symmetric monoid\-al structure on
$\sM(\Gamma)$ and $\gamma^*$: this will be left to the reader.

\subsection{An artificial category}\label{s.art} Recall from Appendix \ref{s.geisser} that the
object $\Z^c$ of $\DM_\et(\F)$ is provided with a multiplication $\mu:\Z^c\oo^L \Z^c\to \Z^c$;
together with the map $\eta:\Z[0]\to \Z^c$, one checks that $\Z^c$ has the structure of a
commutative monoid. This will be better justified in \S \ref{s.w}.

\begin{defn}\label{d10.1} Let $f:S\to \Spec \F$ be a smooth scheme of finite type, and let $A$
be a commutative ring.\\ a) We write $\DA_W(S,A)$ for the category $\DA_\et(S,A)(f^*\Z^c)$, with
the notation of the previous subsection. \\ 
b) If $S=\Spec \F$, we write $\DA_W(\F,A)$ for $\DA_W(\Spec \F,A)$ and
define $\DA_W^\eff(\F)$, $\DA_{\gm,W}(\F)$ and $\DA_{\gm,W}^\eff(\F)$ as the corresponding thick
subcategories.
 \end{defn}

In particular, we have a symmetric monoidal functor $\gamma^*:\DA_\et(S,A)\to \DA_W(S,A)$ with
right adjoint $\gamma_*$, and

\begin{lemma}\label{l10.1} There is a unique symmetric monoidal functor 
\[R_l^W:\DA_W(S,\Z_l)\to \hat{D}_\et(S,\Z_l)\]
such that $R_l=R_l^W\circ \gamma^*$.
\end{lemma}

\begin{proof} The definition of $R_l^W$ on objects is forced since $\gamma^*$ is the identity
on objects; we need to define $R_l^W(f)$ for $f\in \DA_W(S,\Z_l)(M,N)$. Thus $f$ is a morphism
from $M$ to $N\otimes f^*\Z^c$  in $\DA_\et(S,\Z_l)$. We then get
\[R_l(f):R_l(M)\to R_l(N\otimes f^*\Z^c)\simeq R_l(N)\otimes f^*R_l(\Z^c).\]

But $R_l(\eta):\Z_l[0]\to R_l(\Z^c)$ is an isomorphism, in view of the cohomology sheaves of
$\Z^c$;  this gives the desired morphism.
\end{proof}

One interest of the categories $\DA_W(S,A)$ is that they allow us to reformulate the equivalent
conditions of Corollary \ref{c9.1} in a way which justifies the title of this paper:

\begin{prop}\label{p10.1} a) The conditions of Corollary \ref{c9.1} a) are also equivalent to
the following: the restriction of $R_l^W$ to $\DA_{\gm,W}(\F,\Z_l)$ is fully faithful.\\ b) If
they are verified, then the same is true for the restriction of $R_l^W$ to strongly dualisable
objects of $\DA_W(S,\Z_l)$, for any $S\in \Sm(\F)$.
\end{prop}

\begin{proof} a) By part b) of Corollary \ref{c9.1}, the conditions of its part a) imply the
said full faithfulness. Conversely, if we specialise to motives $M(X)$ and $\Z(n)[i]$ for $X$
smooth and $n\ge 0$, $i\in \Z$, we get back Condition (iii) of Corollary \ref{c9.1} a).

b) This follows similarly from Corollary \ref{c9.2}.
\end{proof}

Recall the generator $e\in H^1_\et(\F,\Z^c)$ of \eqref{eq6.1a}. We may view it as a map $e\in
\DA_W(\F)(\Z,\Z[1])$; it is easy to see that $e^2=0$, for example by using its compatibility with
the generator of \eqref{eq6.1c}. Therefore, for any $M,N\in
\DA_W(\F)$, we have a complex of abelian groups
\begin{equation}\label{eq10.5}
\dots \by{\cdot e} \DA_W(S)(M,N[j])\by{\cdot e}\DA_W(S)(M,N[j+1])\by{\cdot
e} \dots 
\end{equation}

\begin{lemma}\label{l10.2} The complex \eqref{eq10.5} has torsion cohomology groups.
\end{lemma}

\begin{proof} It is equivalent to show that the complex
\[
\dots \by{\cdot e} \DA_\et(S,\Q)(M,N\otimes \Q^c[j])\by{\cdot e}\DA_\et(S,\Q)(M,N\otimes \Q^c[j+1])\by{\cdot
e} \dots 
\]
is acyclic. Inspection shows that the map
\[e:\Q[0]\oplus \Q[-1]\simeq \Q^c\to \Q^c[1]\simeq \Q[1]\oplus \Q[0]\]
inducing the differentials is the identity on $\Q[0]$ and $0$ on the other summands. The claim then follows from a simple computation.
\end{proof}

\subsection{Two pairings} We now imitate \S \ref{s.ladic}. Let $M\in \DA_W(\F)$. Composition of
morphisms defines a pairing
\begin{equation}\label{eq10.1}
\Hom(\Z,M)/\tors\otimes \Hom(M, \Z[1])/\tors\to \Hom(\Z,\Z[1])= \Z[1/p].
\end{equation}

Here the value of the target group comes from \eqref{eq6.1a} (recall that $\DM_\et(\F)$, hence
$\DA_W(\F)$, is $\Z[1/p]$-linear).

Thanks now to \eqref{eq7.2}, we have a canonical isomorphism
\[\Hom(\Z,\Q/\Z[1])\simeq (\Q/\Z)'\]
(given by the choice of the arithmetic Frobenius of $\F$). Here we wrote $(\Q/\Z)'$ for $\bigoplus_{l\ne p} \Q_l/\Z_l$. Thus we get a pairing
\begin{equation}\label{eq10.2}
\Hom(\Z,M)\otimes \Hom(M, \Q/\Z[1])\to (\Q/\Z)'.
\end{equation}

The tautological exact sequence
\begin{multline*}
0\to \Hom(M, \Z[1])\otimes \Q/\Z \to \Hom(M, \Q/\Z[1])\\
\to \Hom(M, \Z[2])_\tors\to 0
\end{multline*}
shows that \eqref{eq10.2} induces a pairing
\begin{equation}\label{eq10.3}
\Hom(\Z,M)_\tors\otimes \Hom(M, \Z[2])_\tors\to (\Q/\Z)'.
\end{equation}

\begin{lemma}\label{l.cell} Let $R$ be a commutative ring and $A\times B\to R$ be a
pairing of flat $R$-modules.\\
a) Suppose that the pairing is non-degenerate after $\otimes R_l$ for some prime ideal $l$ of
$R$, where $R_l$ is the completion of $R$ at $l$. Then it is non-degnenerate. \\ 
b) Assume that $R$ is a Noetherian domain with quotient
field $K$. If moreover $A\otimes_R K$ or $B\otimes_R K$ is a finite-dimensional $K$-vector
space, then $A$ and $B$ are finitely generated.\\\
c) In the case of b), suppose moreover that the pairing is perfect after $\otimes R_l$ for every
$l$. Then it is perfect.
\end{lemma} 

\begin{proof} a) and b) are exactly the contents of \cite[Lemma 3.9]{cell}: we refer to loc.
cit. for the proof. The latter involves the composition
\[A_l\by{a} \Hom_R(B,R)_l\by{b} \Hom_{R_l}(B_l,R_l)\]
where $A_l=A\otimes_R R_l$, etc. Since $B$ is finitely generated and $R_l$ is flat over $R$
\cite[II.3.4, Th. 3 (iii)]{bbki}, $b$ is injective (ibid., I.2.10, Prop. 11), thus if $ba$ is
bijective so is $a$. If this is true for all $l$, then $A\to \Hom(B,R)$ is bijective by the same
proof as \cite[II.3.3, Th. 1]{bbki}. This proves c).
\end{proof}

\begin{thm}\label{t10.1} The conditions of Corollary \ref{c9.1} are also equivalent to the
following:\footnote{For the reason why $2$ is inverted, see Theorem \ref{t5.1}.}
\begin{thlist}
\item $\DA_W(\F,\Z[1/2])(M,\Z)$ is a finitely generated $\Z[1/2p]$-module  for
any $M\in\DA_{\gm,W}(\F,\Z[1/2])$.
\item $\DA_W(\F,\Z[1/2])(M,N)$ is a finitely generated $\Z[1/2p]$-module for any
$M,N\allowbreak\in
\DA_{\gm,W}(\F,\Z[1/2])$.
\item For some prime $l\ne 2,p$, $\DA_W(\F,\Z_l)(M,N)$ is a finitely generated $\Z_l$-module for
any $M,N\in \DA_{\gm,W}(\F,\Z_l)$.
\end{thlist}
These conditions imply:
\begin{thlist}
\item[(*)] For
any $M\in\DA_{\gm,W}(\F,\Z[1/2])$, \eqref{eq10.1} is a perfect pairing of finitely generated
free
$\Z[1/2p]$-modules and \eqref{eq10.3} is a perfect pairing of finite groups.
\end{thlist}
\end{thm}

\begin{proof} Let
(**) the condition of Proposition \ref{p10.1} a): we show (**) $\Rightarrow$ (i) + (*)
$\Rightarrow$ (ii)  $\Rightarrow$ (iii)  $\Rightarrow$(**). 

(**) $\Rightarrow$ (i) + (*): we first observe that
$R_l(M)\in \hat{D}^b_c(\F,\Z_l)$ for any
$l\ne p$. Indeed, this is true for $M=M(X)$,
$f:X\to \Spec \F$ smooth by Deligne's finiteness theorem for $l$-adic cohomology (recall that
$R_l(M(X))=(Rf_*\Z_l)^*$). Let $A=\DA_W(\F,\Z[1/2])(M,\Z)$. Clearly, (**) implies
that $A\otimes\Z_l$ is finitely generated over $\Z_l$ for any $l\ne p$. Further, Corollary
\ref{c.verdier} implies the conclusion of (*) after $\otimes\Z_l$.  Lemma
\ref{l.cell} then implies that $A/A_\tors$ is a finitely generated free $\Z[1/p]$-module, that
\eqref{eq10.3} is perfect and that $A\{l\}$ is finite for all $l\ne p$.

To get (i), we observe that the
full subcategory of $\DA_{\gm,\et}(\F,\Z[1/2])$ formed of those $M$ such that
$\DA_W(\F,\Z[1/2])(M,\Z[i])$ is finitely generated for all $i\in\Z$ is triangulated and thick
(stable under direct summands). Thus we may reduce to $M=M(X)(-n)$ with $X$ smooth projective
and $n\ge 0$. In this case, it is known that
\[\hat{D}^b_c(\F,\Z_l)(R_l(M(X)(-n)),\Z_l[i])=H^i_\cont(X,\Z_l(n))\] 
is torsion-free for almost all $l$
\cite[Th. 2]{ctss}. Hence $\DA_W(\F,\Z[1/2])(M,\Z[i])$ has finite torsion, hence is finitely
generated, for any $i$. This also finishes to prove (*).

(i) $\Rightarrow$ (ii): this is clear from the formula 
\[\DA_W(\F,\Z[1/2])(M,N)\simeq\DA_W(\F,\Z[1/2])(N^*\otimes M,\Z).\]

(ii) $\Rightarrow$ (iii): this is obvious.

(iii) $\Rightarrow$ (**): let $M,N\in \DA_{\gm,W}(\F,\Z_l)$. By hypothesis, we have a
homomomorphism of finitely generated $\Z_l$-modules:
\begin{equation}\label{eq10.4}
\DA_W(\F,\Z_l)(M,N)\to D_c(\F,\Z_l)(R_l(M),R_l(N)).
\end{equation}

To work in a triangulated category, we place ourselves in $\DM_\et(\F,\Z_l)$ and rewrite the
above as
\[
\DM_\et(\F,\Z_l)(M,N\otimes \Z^c)
\to
\DM_\et(\F,\Z_l)(M,\Omega_lR_l(N)).
\]

Let $C(N)$ be a cone of the map $N\otimes \Z^c\to
\Omega_lR_lN$ in $D_\et(\F)$. By the same argument as in the proof of Lemma
\ref{l7.1}, $C(N)=0$ if $N$ is torsion, which implies that $C(N)$ is uniquely divisible in
general. Then the exact sequence
\begin{multline*}
\DM_\et(\F,\Z_l)(M,\Omega_lR_l(N))\to \DM_\et(\F,\Z_l)(M,C(N))\\
\to \DM_\et(\F,\Z_l)(M,N\otimes \Z^c[1]) 
\end{multline*}
shows that the $\Q_l$-vector space $\DM_\et(\F,\Z_l)(M,C(N))$ is an extension of two finitely
generated $\Z_l$-modules, hence is $0$. Applying this to $N[-1]$ instead of $N$, we get the
bijectivity of \eqref{eq10.4}.
\end{proof}

\begin{remark}  To control the torsion in Hom groups, we used \cite[Th. 2]{ctss} which
rests ultimately on a theorem of Gabber on the torsion in the $l$-adic cohomology of smooth
projective varieties \cite{gabnote}. This ingredient is not in the spirit of the rest of the
paper: is there a way to avoid it?
\end{remark}

\begin{cor}\label{c10.1}
Suppose that the conditions of Corollary \ref{c9.1} a) hold.  Then,
for any $M,N\in\DA_{\gm,W}(\F,\Z[1/2])$, the complex \eqref{eq10.5} is bounded and has finite
cohomology groups.
\end{cor}

\begin{proof} It follows from Lemma \ref{l10.2} and Theorem \ref{t10.1} that, under the conditions of Corollary \ref{c9.1} a), the cohomology groups of \eqref{eq10.5} are finite. It remains to see its boundedness. In terms of
$\DA_\et(\F,\Z[1/2])$,  amounts to the vanishing of
$\DA_\et(\F,\Z[1/2])(M,N\otimes\Z^c[j])$ for $|j|\gg 0$. By duality, we may reduce to $N=\Z$.  The vanishing is stable under cones and
direct summands, so it suffices to check it on motives of the form
$M=M(X)(-n)$, $X\in
\Sm(\F)$. This amounts to the vanishing of the cohomology groups $H^i_\et(X,\Z^c(n))$
for $|j|\gg 0$. In view of the cohomology sheaves of $\Z^c$, this reduces to the same vanishing
for $H^j_\et(X,\Z(n))$; since $X$ has finite \'etale cohomological dimension, this in turn
reduces in view of \eqref{eq3.1} to the same vanishing for $H^j_\et(X,\Q(n))$. This group is
$0$ for $j>2n$ by \eqref{eq4.3} and Theorem \ref{t3.3} b). To get the boundedness below we may
even reduce to $X$ smooth projective, and then $H^j_\et(X,\Q(n))=0$ for $j<2n$ by the
assumption and Theorem \ref{tG}.
\end{proof}

In \cite{zetavoe}, we shall use Corollary \ref{c10.1} to give a conjectural expression for the
special values of the zeta function of $M$, generalising Lichtenbaum's conjecture up to a power
of $p$ (cf. \cite[Th. 72]{handbook}).

\subsection{Rational coefficients}\label{s.curious}

\begin{prop} \label{p10.2} The following conditions are equivalent:
\begin{thlist}
\item The Beilinson and Parshin conjectures hold for any smooth projective $\F$-variety.
\item For any $M\in \DM_\gm(\F,\Q)$, the composition pairing
\[\DM_\gm(\F,\Q)(\Z,M)\otimes \DM_\gm(\F,\Q)(M,\Z) \to \DM_\gm(\F,\Q)(\Z,\Z)=\Q\]
is non-degenerate.
\end{thlist}
If this is the case,\\
a) The Hom groups in $\DM_\gm(\F,\Q)$ are finite-dimensional $\Q$-vector
spaces. \\
b) The functor $\Phi$ of \eqref{eq8.1} induces an equivalence of categories
\[D^b(\sM_\num(\F,\Q))\iso \DM_\gm(\F,\Q).\]
c) There is a unique $t$-structure on $\DM_\gm(\F,\Q)$ such that $S[-w]$ is in the heart for
any $S\in \sM_\num(\F,\Q)$ of pure weight $w$.\\ 
d) For any $M\in \DM_\gm(\F,\Q)$, the action of
Frobenius on
$R_l(M)$ is semi-simple and its characteristic polynomial is independent of $l$. 
\end{prop}

\begin{proof} The full category of $\DM_\gm(\F,\Q)$ consisting of those $M$ such that the
hypothesis of (ii) is satisfied for $M[i]$ for any $i\in\Z$ is  thick and triangulated. Hence
(ii) is equivalent to the same statement restricted to $M=M(X)(-n)[-i]$, $X$ smooth projective,
$n\ge 0$, $i\in \Z$. Then the equivalence between (i) and (ii) (for this specific $X$) is an
observation essentially due to Jakob Scholbach:

In this case, the said pairing is of the form
\[H_i(X,\Q(n))\times H^i(X,\Q(n))\to \Q.\]

By Poincar\'e duality, the left group may be rewritten $H^{2d-i}(X,\Q(d-n))$.

Suppose $i=2n$. Then we find the intersection pairing between rational Chow groups. To say
that it is nondegenerate is equivalent to say that rational and numerical equivalences agree.

Suppose $i\ne 2n$. Then either $i>2n$ or $2d-i>2(d-n)$. So one of the two groups in the
pairing is $0$, and (ii) is equivalent to the vanishing of the other, which is Parshin's
conjecture.

a) follows from (i)  by internal duality and reduction to smooth projective varieties.
For b), see the proof of \cite[Th. 56]{handbook}. By \cite{jannsen}, $\sM_\num(\F,\Q)$ is
semi-simple and by \cite{km}, any simple $S\in \sM_\num(\F,\Q)$ has a weight; then c) follows
readily from b). Finally, d) follows from b) and Theorem
\ref{rh}, since $M$ is a direct sum of shifted simple motives.
\end{proof}

\begin{remarks} 1) Proposition \ref{p10.2} d) may be applied in particular to the motives
$M(X)$ and $M_c(X)$ of Definition \ref{d6.1}, for $X$ of finite type over $\F$.

2) If we consider $\DA_W(\F,\Q)$, the isomorphism stemming from
\eqref{eq6.1b}
\[\DA_W(\F,\Q)(M,N)\simeq \DM(\F,\Q)(M,N)\oplus \DM(\F,\Q)(M,N[-1])\]
easily shows that \eqref{eq10.1} boils down to the pairing of Proposition \ref{p10.2} (ii).

3) Let $l\ne p$. Using Theorem \ref{t3.2} b) and the duality of Proposition \ref{p.verdier}, it
is easy to see that one has a perfect pairing for any $M\in \DA_{\gm,W}(\F,\Z/l)$:
\[\Hom(\Z,M)\times \Hom(M,\Z[1])\to \Hom(\Z,\Z[1]) = \Z/l.\]

This suggests that the conditions of Proposition \ref{p10.2} might be equivalent to those of
Theorem \ref{t10.1}. Unfortunately a proof of this fails, as we shall see in  Remark
\ref{r.fails}.
\end{remarks}

\section{Weil-\'etale reformulation}\label{s.w}

\subsection{Lichtenbaum's Weil-\'etale topology and cohomology}

\begin{defn}[Lichtenbaum \protect{\cite{licht}}] a) Let $X/\F$ be a scheme of finite type, and
let $\bar X=X\otimes_\F \bar \F$. A
\emph{Weil-\'etale sheaf} on $X$ is 
\begin{itemize}
\item a sheaf $\sF$ on $\bar X$;
\item a map $\phi_\sF:\phi^*\sF\to \sF$, $\phi$ the Frobenius of $\F$.
\end{itemize}
 Weil-\'etale sheaves form the \emph{Weil-\'etale topos} of $X$. There is an underlying site 
$X_W$ described in \cite[\S 2]{licht}.\\ 
b) If $\sF$ is an abelian Weil-\'etale sheaf, one defines
\[H^0_W(X,\sF):= H^0_\et(\bar X,\sF)^{\phi_\sF}.\]
c) Weil-\'etale cohomology is defined as the derived functors of $H^0_W$.
\end{defn}

The definition of the Weil-\'etale cohomology yields short exact sequences
\[0\to H^{i-1}_\et(\bar X,\sF)_\phi\to H^i_W(X,\sF)\to H^i_\et(\bar X,\sF)^\phi\to 0\]
which show that Weil-\'etale cohomology commutes with filtering direct limits of sheaves and has finite cohomological dimension.

There is a morphism $\gamma$ from the Weil-\'etale topos to the \'etale topos (forget that
Frobenius acts continuously), hence adjoint functors: $Ab(X_\et)
\begin{smallmatrix}\gamma^*\\\rightleftarrows\\\gamma_*
\end{smallmatrix}Ab(X_W)$.

\subsection{Two basic results} They are due to Geisser \cite[Th. 5.1 and 5.3]{geisser}:

\begin{thm}\label{tgeisser} a) For any $C\in D(X_\et)$, the map
\[R\gamma_*\Z\oo^L C\to R\gamma_*\gamma^* C\]
of \eqref{eq6.2} is an isomorphism.\\
b) $R\gamma_*\gamma^*\Z \simeq \Z^c$.
\end{thm}

Note that b) is the promised better justification of the monoid structure on $\Z^c$ (see
beginning of \S \ref{s.art})!

Since $D(X_\et)$ is generated by strongly dualisable objects\footnote{namely, locally constant
constructible sheaves} and
$\gamma^*$ preserves compact objects, a) is in fact a formal consequence of Corollary
\ref{c6.1} and Remark
\ref{r9.1}. One will find another proof of b) in Appendix \ref{s.geisser}.

The description of $\Z^c$ given in \S \ref{s6.2} then gives:

\begin{cor}\label{c11.1} a) One has long exact sequences
\[
\dots \to H^i_\et(X,C)\to H^i_W(X,\gamma^*C)
\to H^{i-1}_\et(X,C)\otimes\Q\overset{\partial}{\longrightarrow} H^i_\et(X,C)\to\dots
\]
for $C$ a complex of \'etale sheaves on $X$, where $\partial$ is described by
\eqref{eqpartial}.\\ 
b) If $C$ is torsion, $H^i_\et(X,C)\iso H^i_W(X,\gamma^*C)$.\\
c) If $C$ is a complex of sheaves of $\Q$-vector spaces,
\[H^i_W(X,\gamma^*C)\simeq H^i_\et(X,C)\oplus H^{i-1}_\et(X,C).\]
\end{cor}


\subsection{Inflicting the Weil-\'etale topology on $\DA$} One can faithfully follow the
template of \cite[\S 3]{ayoubreal} and get triangulated categories $\widetilde{\DA}_W(S,A)$ for
$S$ a $\F$-scheme of finite type and $A$ a commutative ring of coefficients. To do this we should define a Weil-\'etale site $\Sm_W(S)$. We can do it by simply mimicking Lichtenbaum's definition in \cite{licht}:

We first define a ``big Weil-\'etale" site $Sch_W(\F)$. Objects of the underlying category are $\bar \F$-schemes of finite type. Let $\bar S$ be such a scheme; for $n\in \Z$ we write $\bar S^{(n)}= (\phi^n)^*\bar S$; there is a canonical $\bar F$-morphism $S^{(n)}\to S$ (iterated relative Frobenius). If $\bar S,\bar T\in Sch_W(\F)$ with $\bar S$ connected, a morphism $\bar S\to \bar T$ is a $\bar \F$-morphism $\bar S\to \bar T^{(n)}$ for some $n\in \Z$; for $\bar S$ arbitrary a morphism from $\bar S$ to $\bar T$ is a collection of such morphisms over the connected components of $\bar S$. Finally, coverings are \'etale coverings.

If $S/\F$ is of finite type, we define $\Sm_W(S)$ as the full subsite of $Sch_W(\F)$ whose objects are smooth $\bar S$-schemes.

One still has a pair
of adjoint functors
\[\DA_\et(S,A)
\begin{smallmatrix}\gamma^*\\\rightleftarrows\\\gamma_*
\end{smallmatrix}\widetilde{\DA}_W(S,A).\]

Here are some good features of this construction:

\begin{lemma}\label{l11.1}
a) $\gamma^*$ is fully faithful if $A=\Z/l^\nu$.\\
b) Theorem \ref{tgeisser} a) remains valid here.\\
c) The functor $\gamma^*$ factors canonically as a composition
\[\DA_\et(S,A)\by{\gamma^*}\DA_W(S,A)\by{\iota} \widetilde{\DA}_W(S,A)\]
where $\DA_W(S,A)$ is the category from Definition \ref{d10.1} a) and $\iota$ is fully
faithful.  Moreover, the functor
$\gamma_*$ of \S \ref{s.art} is the restriction to $\DA_W(S,A)$ of the present functor
$\gamma_*$.\\
d) If $A=\Z/l^\nu$, $\DA_W(S,A)$ is a triangulated subcategory of $\widetilde{\DA}_W(S,A)$.
\end{lemma}

\begin{proof}
a) follows from Corollary \ref{c11.1} b). b) is formal by localisation from Theorem
\ref{tgeisser} a). c) follows formally from b) and Theorem
\ref{tgeisser} b). Finally, d) follows easily from a).
\end{proof}

\begin{remark}\label{r11.1}
Unfortunately $\DA_W(S,A)$ is not a triangulated subcategory of $\widetilde{\DA}_W(S,A)$ in general, \emph{e.g.} for $A=\Z$ or $\Q$. Indeed, consider a fibre $F$ of the map $e\in \widetilde{\DA}_W(\F)(\Z,\Z[1])$ coming from \eqref{eq6.1a} (see above \eqref{eq10.5}). Then $F$ is represented by a Weil-\'etale sheaf (still denoted by $F$), extension of $\Z$ by itself. (One can check that the Frobenius $\phi$ acts on $F\simeq \Z\oplus \Z$ by
\[\phi(m,n)=(m,m+n),\]
compare \cite[(4.1)]{tatesheaf} and Appendix \ref{s.geisser}, but we won't need this here.)

I claim that $F$ is not in the essential image of $\gamma^*$. It is enough to show this for its image $F_\Q$ in $\widetilde{\DA}_W(\F,\Q)$. If $F_\Q$ were of the form $\gamma^*C$ for $C\in \DA_\et(\F,\Q)$, we would have
\[\gamma_* F_\Q \simeq \gamma_* \gamma^* C \simeq C\otimes \gamma_* \Q \simeq C\otimes \Q^c \simeq C\oplus C[1].\]

But it is easy to compute $\gamma_* F_\Q$ as the fibre of $\gamma_* e$: one finds
\[\gamma_* F_\Q\simeq \Q \oplus \Q[1].\]

One would then have $C\oplus C[1] \simeq \Q\oplus \Q[1]$ in $\DA_{et}(\F,\Q)$. Taking cohomology sheaves, this easily gives $C \simeq \Q[0]$.
Hence a contradiction, since then $F_\Q = \gamma^* C = \Q[0]$ would be of $\Q$-rank $1$.
\end{remark}

\subsection{The Weil-\'etale $l$-adic realisation functor} Remark \ref{r11.1} leads to

\begin{defn} We write $\overline{DA}_W(S,A)$ for the localising subcategory of $\widetilde{\DA}_W(S,A)$ generated by $DA_W(S,A)=\gamma^*\DA_\et(S,A)$. It is a $\otimes$-tri\-ang\-ul\-at\-ed subcategory.
\end{defn}

\begin{lemma}\label{l11.2} If $A=\Z/l^n$, $\overline{\DA}_W(S,A)=\DA_W(S,A)\simeq \DA_\et(S,A)$.
\end{lemma}

\begin{proof} This follows formally from Lemma \ref{l11.1} a).
\end{proof}

\begin{prop}\label{p11.1} Let $l\ne 2$. The functor $R_l^W$ of Lemma \ref{l10.1} extends to $\overline{\DA}_W(S)$, and has a right adjoint $\Omega_l^W$. Similarly with coefficients $\Z_l$.
\end{prop}

\begin{proof} It suffices to do it with $\Z_l$ coefficients. This follows immediately from Ayoub's construction of $R_l$ \cite[Def. 5.6]{ayoubreal}: it is the composition of two functors
\[\DA_\et(S,\Z_l)\to \DA_\et(S,\Z/l^\bullet)\to D_\et(S,\Z/l^\bullet)\]
where $\DA_\et(S,\Z/l^\bullet)$ is the ``profinite" version of $\DA_\et(S,\Z_l)$ described in \cite[\S 5]{ayoubreal}. There is an obvious functor 
\[\DA_\et(S,\Z/l^\bullet)\by{\gamma^*} \overline{\DA}_W(S,\Z/l^\bullet)\]
where the right category is defined similarly, and $\gamma^*$ has a right adjoint $\gamma_*$. These adjoints are clearly compatible with the projections
\[\DA_\et(S,\Z/l^\bullet)\by{s_\et^*} \DA_\et(S,\Z/l^s),\; \overline{\DA}_W(S,\Z/l^\bullet)\by{s_W^*} \overline{\DA}_W(S,\Z/l^s).\]

By \cite[Lemma 5.3]{ayoubreal}, the $(s_\et^*)$ form a conservative set of functors, and the same proof shows that the same holds for the $(s_W^*)$. From this and Lemma \ref{l11.2}, one deduces that the unit and counit of the adjunction $(\gamma^*,\gamma_*)$ are isomorphisms, hence $\gamma^*$ and $\gamma_*$ are quasi-inverse equivalences of categories. The functor $R_l^W$ is then defined as the composition
\[\overline{\DA}_W(S,\Z_l)\to \overline{\DA}_W(S,\Z/l^\bullet)\by{\gamma_*} \DA_\et(S,\Z/l^\bullet)\by{R_l} D_\et(S,\Z/l^\bullet).\]

It remains to show that the image of $R_l^W$ is contained in $\hat{D}_\et(S,\Z_l)$. This is a localising subcategory of $D_\et(S,\Z/l^\bullet)$; since it contains the image of $R_l$ and since $\gamma^*\DA_\et(S,\Z_l)$ generates $\overline{\DA}_W(S,\Z_l)$, $\hat{D}_\et(S,\Z_l)$ also contains the image of $R_l^W$. The existence of $\Omega_l^W$  follows fromTheorem \ref{tBrown}.
\end{proof}

On the other hand, there are too many objects in $\widetilde{DA}_W(S,A)$:

\begin{enumerate}
\item Even if $A=\Z/l^\nu$, $\gamma^*$ is not essentially surjective. For example, the free
$\Z/l^\nu[\phi,\phi^{-1}]$-module of rank $1$ defines an object of
$\widetilde{\DA}_W(\F,\Z/l^\nu)$ which is not in the essential image of $\gamma^*$. In
particular, it is not clear whether $R_l^W$
extends to $\widetilde{\DA}_W(\F)$.
\item $\DA_W(\F,\Q)$ is not dense in $\widetilde{\DA}_W(\F,\Q)$. For example, the sheaf $\Q\langle n\rangle$ where Frobenius
acts by multiplication by $q^n$ has no nonzero Weil-\'etale cohomology if $n\ne 0$ (this
example was given by Geisser). Hence it is right orthogonal to $\DA_{\gm,W}(\F,\Q)$, hence to
$\DA_W(\F,\Q)$.
\end{enumerate}

\subsection{A reformulation of Proposition \ref{p10.1}} The following statement is a bit more
elegant:

\begin{prop} The conditions of Corollary \ref{c9.1} a) are also equivalent to the following:
the unit map
\[\Z_l\to \Omega_l^WR_l^W\Z_l\]
of the adjunction in Proposition \ref{p11.1} is an isomorphism.
\end{prop}

\begin{proof} Using Theorem \ref{tgeisser} a) and the fact that
$\gamma^*\DA_\et(\F,\Z_l)$ generates $\overline{\DA}_W(\F,\Z_l)$, this is formally equivalent to Corollary
\ref{c9.1} a) (vii).
\end{proof}

\subsection{A reformulation of Theorem \ref{t10.1}} Consider the composite functor
\begin{multline*}
D(\Z[1/2p])\simeq d_{\le 0} \DA(\F,\Z[1/2p])\to \DA_\et(\F,\Z[1/2])\\
\to \overline{\DA}_W(\F,\Z[1/2]).
\end{multline*}

It commutes with representable direct sums and preserves compact objects, hence  has a right
adjoint $R\Gamma$ wich in turn has a right adjoint $\Pi$.

The cohomology groups of $R\Gamma(\Z)$ are given by \eqref{eq6.1a}; in particular, we have a
canonical map $R\Gamma(\Z)\to \Z[1]$, which yields a map in $\overline{\DA}_W(\F,\Z[1/2])$
\begin{equation}\label{eq11.1}
\Z[-1]\by{\theta} \Pi(\Z).
\end{equation}

\begin{thm} The conditions of Corollary \ref{c9.1} a) are equivalent to the invertibility of
\eqref{eq11.1}.
\end{thm}

\begin{proof} 
Let $M\in \overline{\DA}_W(\F,\Z[1/2])$. Evaluating \eqref{eq11.1} against $M$ produces a map
\[\overline{\DA}_W(\F,\Z[1/2])(M,\Z[1])\by{\theta_M} D(\Z[1/2p])(R\Gamma(M),\Z).\]

The right hand side fits in an exact sequence
\begin{multline*}
0\to \Ext^1(\overline{\DA}_W(\F,\Z[1/2])(M,\Z[1]),\Z[1/p])\to D(\Z[1/2p])(R\Gamma(M),\Z)\\
\to\Hom(\overline{\DA}_W(\F,\Z[1/2])(M,\Z),\Z[1/p])\to 0
\end{multline*}

One checks that $\theta_M$ induces the pairings \eqref{eq10.1} and \eqref{eq10.3}, and that its
bijectivity is equivalent to their perfectness (compare proof of Corollary \ref{c.verdier}).
\end{proof}

\begin{remark}\label{r.fails} As in \S \ref{s.curious}, one easily sees that, in
\eqref{eq11.1}, $\theta\otimes \Z/m$ is an isomorphism for any $m\ne 0$. Hence $\theta$ is an
isomorphism if and only if $\theta\otimes \Q$ is an isomorphism. The problem to go further is
that the map
\[\Pi(\Z)\otimes \Q\to \Pi(\Q)\]
is not \emph{a priori} an isomorphism. It will be, provided $\Pi$ commutes with homotopy
colimits, which happens if and only if $R\Gamma$ preserves compact objects (Lemma
\ref{lAyoub}), in other words, if $\Hom(M,\Z[i])$ is finitely generated and $0$ for almost all
$i$, for any $M\in \DA_W(\F)$\dots

The isomorphisms $\Tor(\Z/m,\Z/m)\allowbreak\simeq \Z/m$ provide a pairing in $D(\Z)$
\[\Q/\Z\times \Q/\Z\to \Q/\Z[1].\]

This pairing  is perfect after $\otimes \Q$ (in $D(\Q)$) and after $\oo^L\Z/m$ (in $D(\Z/m)$)
for any $m$. But it is clearly not perfect, since $\Hom(\Q/\Z,\Q/\Z)$ is not a torsion group.
\end{remark}

\section{A positive case of the conjectures} 

\subsection{Motives of weight $\le 1$} Write
\[w_{\le 1}\DM_{\gm,\et}^\eff(\F)\]
for the thick subcategory of $\DM_{\gm,\et}^\eff(\F)$
generated by the $M(X)$ with $X$ smooth of dimension $0$ and the $\Phi(h_1(C))$ where $C$ is a
smooth projective curve and $\Phi$ is the integral version of the functor in \eqref{eq8.1}.
(Note that
$C$ always has a zero-cycle of degree
$1$ by Weil's theorem, hence an integral Chow-K\"unneth  decomposition.) 

\subsection{A positive result} The aim of this
section is to prove:

\begin{thm}\label{t4ab} If
$M\in w_{\le 1}\DM_{\gm,\et}^\eff(\F)$, \eqref{eq6.4} is an isomorphism.
\end{thm}

\begin{proof} By Lemma  \ref{l7.1}, it suffices to prove this after tensoring with $\Q$. For
simplicity, we drop $\Phi$ from the notation: we reduce
to the case $M= M(Y)$, $Y$ smooth (projective) of dimension $0$, or
$M=h_1(C)$, $C$ a smooth projective curve. The case of an Artin motive is trivial and left to
the reader. As we shall see, that of $h_1$ of a curve follows essentially from Theorem \ref{rh} 
plus Tate's theorem on homomorphisms of abelian varieties
\cite{tate}.

It is equivalent to show that the morphism of Corollary \ref{c9.1} b) is an isomorphism for any
$N$ when $M=h_1(C)$ and after tensoring with $\Q$. We reduce as usual to $N=M(X)[-i]$,
$X$ smooth projective and $i\in \Z$.

Observe that the $l$-adic realisation of $h_1(C)$ is $T_l(J)[1]$, where $J$ is the Jacobian
of $C$. Thus we want to show that the homomorphism
\begin{equation}\label{1/2tate}\DM_{\gm,\et}^\eff(k)(M(X),\Gamma^\eff \otimes
h_1(C)[i-1])\otimes \Q\to  H^i_\cont(X,V_l(J))
\end{equation}
is an isomorphism for all $i\in \Z$.

By duality, the left hand side of \eqref{1/2tate} is isomorphic to
\[\DM^{\et,\eff}_\gm(k)(M(X)\otimes h_1(C),\Gamma^\eff(1)[i+1])\otimes \Q\]
which is a direct summand of
\[H^{i+1}_\et(X\times C,\Gamma^\eff(1))\otimes \Q\simeq H^{i+1}_\et(X\times C,\Z(1))\otimes
\Q_l\oplus H^{i+2}_\et(X\times C,\Z(1))\otimes \Q_l\]
cf. \eqref{eq6.1b} and \eqref{eq7.1}. We have
\[H^{i}_\et(X\times C,\Z(1))\otimes \Q_l=
\begin{cases}
\Pic(X\times C)\otimes \Q_l&\text{if $i=2$}\\
0 &\text{otherwise.}
\end{cases}
\]

Using the isomorphism (theorem of the cube $+$ finiteness of $\Pic^0(X\times C)$)
\[\Pic(X\times C)\otimes \Q_l\simeq \NS(X)\otimes \Q_l \oplus \NS(C)\otimes \Q_l \oplus
\Hom(\Alb(X),J)\otimes
\Q_l\] 
we see that the direct summand $\DM^\et(k)(M(X)\otimes
h_1(C),\Z(1)[2])\otimes \Q_l$ of $H^{2}_\et(X\times C,\Z(1))\otimes \Q_l$ gets identified with
$\Hom(\Alb(X),J)\otimes \Q_l$. 

A weight computation shows that the right hand side of \eqref{1/2tate} is
\[H^i_\cont(X,V_l(J))=
\begin{cases}
H^1_\cont(\bar X,V_l(J))^G &\text{si $i=1$}\\
H^1_\cont(\bar X,V_l(J))_G &\text{si $i=2$}\\
0 &\text{otherwise.}
\end{cases}
\]

Moreover, the action of $G$ on $H^1_\cont(\bar
X,V_l(J))\simeq \Hom(V_l(\Alb(X)),V_l(J))$ is semi-simple, hence the composition
\[H^1_\cont(\bar X,V_l(J))^G\to H^1_\cont(\bar X,V_l(J))\to H^1_\cont(\bar X,V_l(J))_G\]
is an isomorphism; by Lemma \ref{l4.2}, this isomorphism
may be identified with cup-produit by $e\in H^1_\cont(k,\Q_l)$:
\[H^1_\cont(X,V_l(J))\by{e}H^2_\cont(X,V_l(J)).\]

Poincar\'e duality now realises $H^i_\cont(X,V_l(J))$ as a direct summand of
$H^{i+1}_\cont(X\times C,\Q_l(1))$.

In \eqref{1/2tate}, the two sides are therefore $0$ for $i\ne
1,2$; for $i=1,2$ they fit in a commutative diagram
\[\begin{CD}
\Hom(\Alb(X),J)\otimes \Q_l @>u>>  \Hom(V_l(\Alb(X)),V_l(J))^G\\
|| && @V{\wr}VV\\
\Hom(\Alb(X),J)\otimes \Q_l @>v>>  \Hom(V_l(\Alb(X)),V_l(J))_G.
\end{CD}\]

By Tate's theorem \cite{tate}, $u$ is an isomorphism. The diagram shows that so is $v$, hence
\eqref{1/2tate} is indeed an isomorphism for any $i\in\Z$.
\end{proof}

\appendix

\section{A letter to T. Geisser}\label{s.geisser}

\hfill Paris, May 2, 2004.

\bigskip
Dear Thomas,

To start with, let $\Gamma$ be a group, $A$ a $\Gamma$-module and $e\in
H^1(\Gamma,A)$ a class. To $e$ corresponds an extension of $\Gamma$-modules
\[0\to A\to A'\to \Z\to 0.\]

It is an exercise of homological algebra to check
that $A'$ may be described\footnote{Up to sign, as in all this letter.} as $\Z\times A$ (the underlying abelian group)
with $\Gamma$-action given by
\[a(r,s)=(r,re(a)+ as).\]

This can be used as follows. Given a complex of $\Gamma$-modules $C$, cup-product by $e$
\[\cdot e: \bH^i(\Gamma,C)\to \bH^{i+1}(\Gamma, A\oo^L C)\]
may be described, up to sign, as the boundary morphism $\partial_C$ of the exact triangle
\[A\oo^L C\to A'\oo^L C\to C\overset{\partial_C}{\to} A\oo^L C[1].\]

(This is another exercise of homological algebra: observe that $e=\partial_\Z(1)$ and  compare
\cite[Ch. VIII, \S 3, Prop. 5]{cl}).

All this can be [made] topological. In particular, for $\Gamma=\hat\Z$,
$A=\hat\Z$ and $e=e$ we get a corresponding $A'=\hat{M}$ such that
$\hat{M}\otimes\Q/\Z=\tilde M$ (notation from \cite[\S 4]{tatesheaf}).

This first justifies the claim in loc. cit., Prop. 4.4 that the bottom row
in the diagram is cup-product by $e$. Second, it gives
another way to look at $M$: it is the push-out of the extension
\begin{equation}\label{e0}
0\to \hat \Z\otimes\Q\to \hat{M}\otimes\Q\to\Q\to 0
\end{equation}
by the projection $\hat \Z\otimes\Q\to \hat \Z\otimes\Q/\Z=\Q/\Z$. Viewed
from this side, we may recover $\gamma:\Q\to M$ as the composite $\Q\to
\hat \Z\otimes\Q\to \Q/\Z\to M$.

Another thing you asked me is to explain the proof of \cite[Cor. 4.8]{tatesheaf}. First, by
Theorem 4.6, the map
\[\bH^1(\F_p,\Z^c)\otimes \Z_l=\bH^1(\F_p,\Z^c\otimes \Z_l)\to \bH^1(\F_p,\Z_l(0)^c)=H^1_\cont(\F_p,\Z_l)\]
is an isomorphism and therefore sends a generator $e_0$ of $\bH^1(\F_p,\Z^c)$ to $ue$, where $u$ is an $l$-adic unit. To prove that $u=1$, it suffices to go mod $l^\nu$ for all $\nu$ and to prove that the induced map
\[\Z=\bH^1(\F_p,\Z^c)\to H^1(\F_p,\Z/l^\nu)=\Z/l^\nu\]
is the projection. For this, observe that by definition of $\tilde M$ and $M$, we have a short exact sequence
\begin{equation}\label{e1}
0\to \Z\to M\to \tilde M\to 0
\end{equation}
and that the map $\Z\to M$ splits the exact sequence of \cite[Lemma 4.3]{tatesheaf}. This
implies that $e_0\in H^1(\hat\Z,\Z^c)$ is the image of the generator of $H^0(\hat
\Z,\Z)=H^1(\hat\Z,\Z[-1])$ under the map of complexes
\begin{equation}\label{e2}
\Z[-1]\to \Z^c
\end{equation}
defined by \eqref{e1}. Another way to formulate this is that the morphism defined by $e_0$ in the derived category is realised by \eqref{e2}. To prove the claimed compatibility, we now have to identify the cone of \eqref{e2} tensored with $\Z/l^\nu$ with $\hat{M}\oo^L\Z/l^\nu$. This cone is $[\Q\to \tilde M]$; we have a string of quasi-isomorphisms
\[[\Q\to \tilde M]\oo \Z/l^\nu\iso \tilde M\oo^L\Z/l^\nu[-1]\simeq \hat{M}\oo^L\Q/\Z\oo^L\Z/l^\nu[-1]\simeq \hat{M}\oo^L\Z/l^\nu.\]

Then for the second claim of Corollary 4.8. First note that the statement is nonsense. In fact, the splitting $\Q^c\simeq \Q\oplus\Q[-1]$ of Corollary 4.5 defines a ``projection onto $H^0$" $\Q^c\to \Q$ and an ``inclusion of $H^1$" $\Q[-1]\to \Q^c$. (Note that the only endomorphism of $\Q\oplus \Q[-1]$ that induces the identity on cohomology is the identity because $Ext^1_{\hat\Z}(\Q,\Q)=0$, so these projection and inclusion are unambiguous.) The correct statement is now that the corresponding composition
\[\Q^c\to \Q\to \Q^c[1]\]
is given by cup product by $e$.

In order to make sense of this claim, I should define a product
\[\Q^c\oo^L\Q^c\to \Q^c\]
matching the product
\[\Q_l(0)^c\oo^L\Q_l(0)^c\to \Q_l(0)^c\]
which corresponds to the ``continuous" product $\Q_l\otimes\Q_l\to \Q_l$. There is a stupid way to do this, and then one can use the isomorphism of Theorem 4.6 b) tensored with $\Q$ and the obvious description of cup-product by $e$ on $\Q_l(0)^c$: this is ugly but it works, and rationally this is all there is to it. Nevertheless let me do all this integrally as it might be useful. We have a short exact sequence of topological modules over $\hat{\Z}$:
\[0\to \hat\Z\to \hat M\otimes\Q\to M\to 0\]
hence a flat version of $\Z^c$ is given by the length $1$ complex
\begin{equation}\label{e3}
[\Q\oplus \hat\Z\to \hat M\otimes\Q]
\end{equation}
(note that, as already remarked, the map $\gamma:\Q\to M$ lifts canonically to $\hat{M}\otimes
\Q$ via \eqref{e0}.) This gives us a version of $\Z^c\oo^L\Z^c$:
\[[\Q\oplus (\Q\otimes\hat \Z)^2\oplus \hat \Z\to (\hat M\otimes \Q)^2\to (\hat
M\otimes\Q)^{\otimes 2}].\]

We define the product $\Z^c\oo^L\Z^c\to \Z^c$ by the cochain map which is the projection on
$\Q\oplus\hat\Z$ in degree $0$ and the sum in degree $1$. With this definition, the diagrams
\[\begin{CD}
\Z^c\oo^L\Z^c@>>> \Z^c\\
@VVV @VVV\\
\Z/l^\nu\oo^L\Z/l^\nu@>>>\Z/l^\nu
\end{CD}\]
obviously commute coherently (for solid versions of the maps obtained by na\"\i vely tensoring
the flat version of $\Z^c$ by $\Z/l^\nu$), which proves the compatibility with the product
$\Z_l(0)^c\oo^L\Z_l(0)^c\to \Z_l(0)^c$. If we now tensor with $\Q$, we find that on $\Q^c\simeq
\Q\oplus \Q[-1]$, this product is given by the identity $\Q\to \Q$ in degree $0$ and the sum
$\Q\oplus\Q\to \Q$ in degree $1$. From this and the description \eqref{e2} of the first map of
the composition
\[\Q^c[-1]\overset{\cdot e_0}{\longrightarrow} \Q^c\oo^L\Q^c\to \Q^c\]
you easily deduce the second claim of \cite[Cor. 4.8]{tatesheaf}. 

You might also be interested in the following construction of an isomorphism $\Z^c\iso
R\gamma_*\Z$. By adjunction I need to give a map $\phi:\gamma^*\Z^c\to\Z$. Now to describe
$\gamma^*\Z^c$, I may use a discrete version of the topological $\hat\Z$-module $\hat M$, say
$\bar M$:
\[
0\to\Z\to\bar M\to\Z\to 0.
\]

Then a flat version of $\gamma^*\Z^c$ is given by
\begin{align*}
\Q\oplus\Z&\to \bar M\otimes\Q \\
(r,s)&\mapsto (0,s-r)
\end{align*}
(the discrete version of \eqref{e3}), and there is an obvious $\Z$-equivariant map $\phi$ from
this complex to $\Z[0]$ ($(r,s)\mapsto s$). You will easily check by considering cones that the
diagram
\[\begin{CD}
\Z@=\Z\\
@V{\gamma^* e_0}VV @V{e_1}VV\\
\gamma^*\Z^c[1]@>\phi>> \Z[1]
\end{CD}\]
commutes, where $e_1$ is the generator of $H^1(\Z,\Z)=Ext^1_{\Z[\Z]}(\Z,\Z)$. By adjunction it
follows that the diagram
\[\begin{CD}
\Z@=\Z\\
@V{e_0}VV @V{\tilde e_1}VV\\
\Z^c[1]@>\tilde\phi>> R\gamma_* \Z[1]
\end{CD}\]
commutes, where $\tilde\phi$ and $\tilde e_1$ denote the maps corresponding to $\phi$ and
$e_1$ under adjunction. This implies that $H^1(\hat{\Z},\Z^c)\iso H^1(\hat{\Z},R\gamma_*\Z)$
under $\tilde\phi$; passing to the open subgroups of $\hat\Z$ we get that $\tilde\phi$ induces
an isomorphism on the $\sH^1$s. Since it also obviously induces an isomorphism on the $\sH^0$s,
this proves that it is a quasi-isomorphism.
\enlargethispage*{20pt}

I hope this is useful.

Best regards,
\bigskip

\hfill Bruno

\enlargethispage*{20pt}

\end{document}